\documentclass[sn_mathphys,Numbered]{sn_jnl}%

\usepackage{graphicx}%
\usepackage{multirow}%
\usepackage{amsmath,amssymb,amsfonts}%
\usepackage{amsthm}%
\usepackage{mathrsfs}%
\usepackage[title]{appendix}%
\usepackage{textcomp}%
\usepackage{manyfoot}%
\usepackage{booktabs}%
\usepackage{algorithm}%
\usepackage{graphicx}
\usepackage{float}
\usepackage{mathtools}
\usepackage{wrapfig}
\usepackage{amssymb, amsmath, latexsym}
\usepackage{nicefrac}
\usepackage{dsfont}
\usepackage{hyperref}
\usepackage[title]{appendix}
\usepackage{hhline}
\usepackage{amsthm}
\usepackage{multirow}
\usepackage{pifont}

\usepackage{algorithm}\usepackage{algpseudocode}

\def \RR {\mathbb R}
\def \R {\mathbb R}
\def\eqdef{\overset{\text{def}}{=}}

\usepackage{tikz}
\newcommand*\circled[1]{\tikz[baseline=(char.base)]{
            \node[shape=circle,draw,inner sep=2pt] (char) {#1};}}
\newtheorem{assumption}{Assumption}
\newtheorem{theorem}{Theorem}[section]

\newtheorem{lemma}[theorem]{Lemma}
\newtheorem{definition}{Definition}[section]

\newcommand{\cD}{{\cal D}}

\newcommand{\EE}{\mathbf{E}}

\def\R{\mathbb{R}}

\def\X{\mathcal X}

\def\R{\mathbb R}

\def\EE{\mathbb E}

\usepackage{xcolor}
\def\dd#1{{\color{black}#1}}

\def\dm#1{{\color{black}#1}} 

\begin{document}

\title[Gradient-Free Methods for Non-Smooth Convex Stochastic Optimization with Heavy-Tailed Noise on Convex Compact]{Gradient-Free Methods for Non-Smooth Convex Stochastic Optimization with Heavy-Tailed Noise on Convex Compact}

\author[1]{\fnm{Nikita} \sur{Kornilov}}\email{kornilov.nm@phystech.edu} 
\author[1,4,5]{\fnm{Alexander} \sur{Gasnikov}}\email{gasnikov@yandex.ru} 
\author[2]{\fnm{Pavel} \sur{Dvurechensky}}\email{pavel.dvurechensky@wias-berlin.de} 
\author[3]{\fnm{Darina} \sur{Dvinskikh}}\email{dmdvinskikh@hse.ru}

 \affil[1]{\orgname{Moscow Institute of Physics and Technology}, \orgaddress{\city{Dolgoprudny}, \country{Russia}}}
 \affil[2]{\orgname{Weierstrass Institute for Applied Analysis and Stochastics}, \orgaddress{\city{Berlin}, \country{Germany}}}
 \affil[3]{\orgname{HSE University}, \orgaddress{\city{Moscow}, \country{Russia}}}
\affil[4]{\orgname{Skoltech}, \orgaddress{\city{Moscow}, \country{Russia}}}
\affil[5]{\orgname{ISP RAS Research Center for Trusted Artificial Intelligence}, \orgaddress{\city{Moscow}, \country{Russia}}}

\maketitle
\begin{abstract}
\dm{We present two easy-to-implement gradient-free/zeroth-order methods to optimize a stochastic non-smooth function accessible only via a black-box.
The methods are built upon  efficient first-order methods in the heavy-tailed case, i.e., when the gradient noise  has infinite variance but bounded $(1+\kappa)$-th moment for some $\kappa \in(0,1]$.
The first algorithm is based on the stochastic mirror descent with a particular class
of uniformly convex mirror maps which is robust to heavy-tailed noise.
The second algorithm is based on the stochastic mirror descent and gradient clipping technique. Additionally, for the objective functions satisfying the $r$-growth condition, faster algorithms are proposed based on these methods and the restart technique. }
\end{abstract}

\keywords{zeroth-order optimization \and derivative-free optimization \and stochastic optimization \and non-smooth problems \and heavy tails \and gradient clipping \and stochastic mirror descent}
\maketitle
\section{Introduction}
We consider stochastic non-smooth convex minimization problem 
\begin{equation}
\label{eq:prst}
    \min\limits_{x \in \mathcal{X} \subset \RR^d} f(x)\eqdef  \EE_{\xi\sim\cD}\left[f(x,\xi)\right], 
\end{equation}
 \dm{where function $f(x, \xi)$ is $M_2(\xi)$-Lipschitz continuous  in $x$    w.r.t. the Euclidean norm,
  $\mathcal{X} $ is a compact convex, and  the expectation $\EE_{\xi\sim\cD}\left[f(x,\xi)\right]$ is w.r.t. random variable $\xi$  with unknown distribution $\cD$.
 We suppose that stochastic realizations of the function values $f(x)$ are  available only through a zeroth-order oracle corrupted by some deterministic (probably adversarial) noise $\delta(x)$ }
\begin{equation}\label{eq:grad_free_setup}
      \phi(x,\xi) \eqdef f(x, \xi) +\delta(x). 
\end{equation}
     We consider two-point zeroth-order oracle setting meaning that for two query points $x,y \in \mathcal \X$ we can evaluate two outputs $\phi(x,\xi)$ and $\phi(y,\xi)$ with the same $\xi$. Function $\phi(x,\xi)$ can be considered as  a  noisy approximation of a Lipschitz function $f(x,\xi)$. 

       Stochastic optimization problems implies that functions $f(x, \xi)$  must have finite mathematical expectation for all $x \in \mathcal{X}$.


\dd{Zeroth-order methods were studied in a wide range of works, see e.g., \cite{spall2005introduction,conn2009introduction} and the references therein. Particularly, under different assumptions on black-box oracle (in the noisy or noiseless setup) the optimal oracle complexity was obtained \cite{duchi2015optimal,gasnikov2016gradient,nesterov2017random,gasnikov2017stochastic,shamir2017optimal,bayandina2018gradient}. This bound is proportional to $d \varepsilon^{-2}$, where $\varepsilon$ is the desired precision to solve problem \eqref{eq:prst} in terms of the function values. For saddle point problems, we refer to papers \cite{beznosikov2020gradient,dvinskikh2022gradient} obtaining the same bound.}
This result is quite expected since the above complexity is $d$ times larger than the complexity of optimal stochastic gradient procedures. Factor $d$ has a natural interpretation since to approximate (stochastic) gradient it suffices to use $d + 1$ function values.\footnote{To say more precisely, it suffices to use $d + 1$  values of $f(x,\xi)$ with the same $\xi$ and different $(d+1)$ points $x$.} This is obvious in the smooth case (see e.g. \cite{gasnikov2022randomized}), and is not so trivial in the non-smooth case \cite{shamir2017optimal}. This result was obtained it  the classical  setting of a finite variance of stochastic gradients:   $\EE_\xi [M(\xi)^2] < \infty$. However, in modern learning problems, 
this condition may be violated. To this end, we aim to relax this assumption and consider heavy-tailed noise with bounded $(1+\kappa)$-th moment for some $\kappa \in(0,1]$, i.e.,  we suppose $\EE_\xi [M(\xi)^{1+\kappa}] < \infty$.  Under this assumption, for the first-order stochastic methods, the optimal oracle complexity is proportional to  $\varepsilon^{-\frac{1+\kappa}{\kappa}}$ \cite{nemirovskij1983problem}. Thus for zeroth-order  oracle we may expect the bound $d\varepsilon^{-\frac{1+\kappa}{\kappa}}$.
\dd{In this paper, we obtain} the bound
$\left(\sqrt{d}/\varepsilon\right)^{\frac{1+\kappa}{\kappa}}$  matching the expected bound only for $\kappa = 1$.
To the best of our knowledge, this poses the following open problem: is the bound
$\left(\sqrt{d}/\varepsilon\right)^{\frac{1+\kappa}{\kappa}}$  optimal in terms of the dependence on $d$? For smooth stochastic convex optimization problems with $(d+1)$-points stochastic zeroth-order oracle the answer is negative and the optimal bound is proportional to $d  \varepsilon^{-\frac{1+\kappa}{\kappa}}$. Thus, for $\kappa \in (0,1)$ our results are \dd{somewhat} surprising since the dependence on $d$ in our bound is very different from the known results for the case  $\kappa = 1$. 
To the best of our knowledge, this paper provides the first known result for gradient-free methods without assuming a finite variance of the stochastic noise. Since we give an accurate analysis, including high-probability bounds,\footnote{We emphasize that these bounds were obtained without any probabilistic assumptions, except $\EE_\xi [M^{1+\kappa}(\xi)] < \infty$!} our results could be of interest even in a very particular case of $\kappa = 1$. In this case, the high-probability bound was previously known only for compactly supported distributions of $f(x,\xi)$ \cite{dvinskikh2022gradient}. That is, even for sub-Gaussian tails \cite{shapiro2021lectures} it was an open question to obtain high-probability bounds for gradient-free methods. The main challenge in obtaining our results is in the combination of the auxiliary gradient-free randomization and the original stochasticity of the oracle in the problem. The known  inequalities on measure concentration  do not allow obtaining the desired sub-Gaussian concentration for the output of the algorithm.

  \dd{
  Gradient clipping technique has become increasingly popular for obtaining convergence guarantees in terms of high probability \cite{nazin2019algorithms,davis2021low,gorbunov2021near}.}
 Starting with the work \cite{nazin2019algorithms} (see also \cite{davis2021low,gorbunov2021near}) one can observe an increased interest of researchers in algorithms that use gradient clipping to be able to obtain high-probability convergence guarantees in stochastic optimization problems with heavy-tailed noise.
In particular, only in the last two years \dd{optimal first-order algorithms were proposed} and the following results were obtained \dd{for their convergence guarantees}:  \textbf{1.} in the  expectation for general proximal setup and non-smooth stochastic convex optimization problems with infinite variance \cite{vural2022mirror}; \textbf{2.} in high-probability for general proximal setup and non-smooth online stochastic convex optimization problems with infinite variance \cite{zhang2022parameter}; \textbf{3.}  in  high-probability for the Euclidean proximal setup and  smooth and non-smooth stochastic convex optimization problems and variational inequalities with infinite variance \cite{sadiev2023high,nguyen2023high,nguyen2023improved}; \textbf{4.}   in high-probability for convergence of optimal variance-adaptive algorithm in the Euclidean  proximal setup for non-smooth stochastic (strongly) convex optimization problems with infinite variance \cite{liu2023stochastic}.  Since the aforementioned results  are strongly correlated with each other, in this paper, we depart from the works \cite{vural2022mirror,zhang2022parameter} to incorporate zero-order oracle into their algorithms. The developed technique, which reduces randomization caused by the gradient-free nature of the oracle to the original stochasticity, allows generalizing the results of other papers considered above in a similar manner. The idea of this reduction is not new and has already been used many times, see e.g. \cite{duchi2015optimal,gasnikov2016gradient,gasnikov2017stochastic,shamir2017optimal}. But, all these works are significantly based on the assumption of finite variance of the stochastic noise. For the infinite noise variance setting, the technique requires significant generalizations, which we make in this paper. We expect, that based on these results it is possible to obtain new results for zero-order algorithms in the smooth setting and also in the setting of one-point feedback. 


 \paragraph{Contribution} 
 \dm{ \begin{enumerate}
    \item  For $d$-dimensional
 optimization, we propose two algorithms  with
 oracle complexity proportional to $\left({\sqrt{d}}/{\varepsilon}\right)^{\frac{1+\kappa}{\kappa}}$. This upper bound is valid under the
maximal admissible level of adversarial noise  proportional to $ {\varepsilon^2}/{\sqrt{d}}$. For the first algorithm the convergence results hold  in expectation
whereas for the second algorithm the results are valid  with high probability.
\item  If additionally the objective satisfies the $r$-growth condition ( this includes strongly convex problems and problems with a sharp minimum), the restart technique for these algorithms gives oracle complexity proportional to $\left(  \sqrt{d}/\varepsilon^{\frac{(r-1)}{r}}\right)^\frac{1+\kappa}{\kappa}$. This  upper bound is valid under the
maximal level of adversarial noise  proportional to $ \varepsilon^{(2-\frac{1}{r})} / \sqrt{d}$.
 \end{enumerate}
}

\paragraph{Organization} \dd{ This paper is organized as follows. Section \ref{sec:grad-free} presents the main objects and notions that are used to construct gradient-free algorithms. 
In Section \ref{sec:robust_ZOSMD}, we present  our first gradient-free algorithm which is based on mirror descent. In Section \ref{sec:ZOCLipping} we present our second gradient-free  algorithm based on gradient clipping. Finally, in Section \ref{sec:restarts}  for the objective functions satisfying the $r$-growth condition, we propose a faster algorithm  using the restart technique. }

\section{Preliminaries} \label{sec:grad-free}

\paragraph{Notations}  For $p \in [1,2]$, we use the $l_p$-norm, i.e.  $\|x\|_p = \left(\sum_{k=1}^d |x_k|^p\right)^{1/p}$. The corresponding dual norm is $\|y\|_{q} = \max_x \{\langle x, y \rangle | \;\|x\|_p \leq 1\}$, where $q$ is defined by the equality $\nicefrac{1}{q} + \nicefrac{1}{p} = 1$. 
We use $\langle x,y \rangle = \sum_{k=1}^d x_ky_k$ to denote the inner product of $x,y \in \RR^d$.
Let $B_{\dm{p'}}^d = \{x \in \R^d\mid \|x\|_{\dm{p'}} \leq 1\}$
and $S_{\dm{p'}}^d = \{x \in \R^d\mid \|x\|_{\dm{p'}} = 1\}$ 
be the  unit  $\ell_{\dm{p'}}$-ball
and  the unit $\ell_{\dm{p'}}$-sphere \dm{with center at 0}, correspondingly.  
 The full expectation of a random variable $X$ is denoted by $\EE [X]$. The expectation w.r.t. random variables $Y_1, \dots, Y_n$   is denoted by $\EE_{Y_1, \dots, Y_n} [X]$.
 The condition expectation w.r.t. $x_k, \dots, x_1$ is refereed to as 
 $\EE[\cdot | x_k, \dots, x_1] \eqdef \EE_{|\leq k}[\cdot]$ for brevity.

 \subsection{Assumptions}
For a convex set $\mathcal{X}\subset \R^d$ and $\tau > 0$, let us  introduce $\mathcal{X}_\tau = \mathcal{X} + \tau B^d_2$. 


\begin{assumption}[Convexity] \label{as: convex}
There exists $ \tau > 0$ such that function $f(x, \xi)$ is convex \dd{w.r.t. $x$} for any $\xi$ on $\mathcal{X}_\tau$.
\end{assumption}
This assumption implies that $f(x)$ is  convex on $\mathcal{X}$.
\begin{assumption}[Lipschitz continuity and boundedness of $(1+\kappa)$-th moment] \label{as: Lipshcitz and bounded}
There exists $ \tau > 0$ such that function $f(x, \xi)$ is $M_2(\xi)$-Lipschitz continuous \dd{w.r.t. $x$} in the  $l_2$-norm, i.e., for all $x_1, x_2 \in \mathcal{X}_\tau$
$$|f(x_1, \xi) - f(x_2, \xi)| \leq M_2(\xi) \|x_1 - x_2\|_2.$$
Moreover, there exist $ \kappa \in (0,1] $ \dd{and  $M_2$} such that $\EE_\xi [M_2(\xi)^{1+\kappa}] \leq M_2^{1+\kappa}.$
\end{assumption}
\begin{lemma}\label{lem: Lipschitz f intro}
    Assumption~\ref{as: Lipshcitz and bounded} implies that $f(x)$ is $M_2$-Lipschitz on $\mathcal{X}$.
\end{lemma}
The proof can be found in Section \ref{sec: Lemmas} (Lemma \ref{lem: Lipschitz f }).
\begin{assumption}[Boundedness of noise]\label{as: noize}
There exists a constant $\Delta>0$ such that $|\delta(x)| \leq \Delta$ for all $x \in Q$.
\end{assumption}
\paragraph{Randomized smoothing.}

\dm{The main scheme that allows us to develop gradient-free
methods for non-smooth convex problems 
is randomized smoothing \cite{ermoliev1976stochastic,nemirovskij1983problem,spall2005introduction,nesterov2017random,gasnikov2022power} of a non-smooth function $f(x,\xi)$.
  The smooth approximation to a non-smooth function  $f(x,\xi)$ is
defined as }
\begin{equation}\label{hat_f}
    \hat{f}_\tau (x) \overset{\text{def}}{=} \EE_{\dm{\mathbf{u}\dm{,\xi}} } [f(x + \dm{\mathbf{\tau u}} \dm{,\xi})],
\end{equation}
\dm{where $\mathbf{u} \sim  U(B^d_2)$  is a random vector uniformly distributed on the Euclidean unit ball.}

\dm{The next lemma gives estimates for the quality of this approximation.
In contrast to $f(x)$, function $\hat{f}_\tau(x)$ is smooth  and has several useful properties. 
}



\begin{lemma}\cite[ Theorem $2.1$]{gasnikov2022power}\label{lem: hat_f properties intro}
Let Assumptions~\ref{as: convex},\ref{as: Lipshcitz and bounded} hold. Then,
\begin{enumerate}
\item Function $\hat{f}_\tau(x)$ is convex, Lipschitz with constant $M_2$  on $\mathcal{X}$, and satisfies
\begin{equation}\notag
    \sup \limits_{x \in \mathcal{X}} |\hat{f}_\tau(x) - f(x)| \leq \tau M_2.
\end{equation}
    \item 
Function $\hat{f}_\tau(x)$ is differentiable on $\mathcal{X}$ with the following gradient
\begin{equation}\notag
\nabla \hat{f}_\tau (x) = \EE_\mathbf{e}\left[\frac{d}{\tau} f(x + \tau \mathbf{e}) \mathbf{e}\right],
\end{equation}
\dm{where $\mathbf{e} \sim U(\dm{S^d_2})$ is a random vector uniformly distributed on the 
Euclidean unit sphere. }
\end{enumerate}
\end{lemma}

\dm{\paragraph{Gradient estimate.}
To employ first-order algorithms in the zero-order oracle setting, we 
use the following gradient estimate}
\begin{eqnarray}\label{g}
g(x, \xi, \mathbf{e}) &=& \frac{d}{2\tau}(\phi(x + \tau \mathbf{e}, \xi) - \phi(x - \tau \mathbf{e}, \xi)) \mathbf{e} \notag \\
&=& \frac{d}{2\tau}(f(x + \tau \mathbf{e}, \xi) + \delta(x + \tau \mathbf{e}) - f(x - \tau \mathbf{e}, \xi) - \delta(x - \tau \mathbf{e})) \mathbf{e}.
\end{eqnarray}
We can notice that this vector will be an unbiased estimate of 
the gradient of $\hat{f}_\tau(x)$ if there is no adversarial noise $\Delta = 0$. Moreover, this vector has bounded $(1 + \kappa)$-th moment, see the next lemma.
\begin{lemma}\label{lem: grad norm 1 +k intro}
Under Assumptions \ref{as: convex}, \ref{as: Lipshcitz and bounded} and \ref{as: noize}, for $q \in [2, +\infty)$, we have
\begin{equation*}
    \EE_{\xi, \mathbf{e}}[\|g(x,\xi,\mathbf{e})\|_q^{1+\kappa}]\leq 2^{\kappa}\left(\frac{\sqrt{d}}{2^{1/4}} a_{q}M_2\right)^{1+\kappa} + 2^{\kappa}\left(\frac{da_{q}\Delta}{\tau}\right)^{1+\kappa}  \eqdef \sigma_{q}^{1+\kappa},
    \end{equation*}
where $a_q \eqdef d^{\frac1q - \frac12}  \min \{ \sqrt{32\ln d - 8} , \sqrt{2q - 1}\}.$
\end{lemma}
The proof can be found in the
\dm{Appendix} \ref{sec: Lemmas}.
 \section{ \dm{ First Algorithm: \texttt{ZO-RSMD} }}
\label{sec:robust_ZOSMD}
\dm{In this section, we present our first gradient-free algorithm which is built upon  mirror descent algorithm  
with  uniformly convex mirror map  from \cite{vural2022mirror}. Our algorithm as well as algorithm from \cite{vural2022mirror}  is robust to heavy-tailed noise.   Firstly we provide 
 mirror descent algorithm  
with  uniformly convex mirror map  from \cite{vural2022mirror} and then we present its zeroth-order version.}

\subsection{Robust Stochastic Mirror Descent (\texttt{RSMD})}
\label{sec:robustSMD}
 \dm{ Now we present convergence results for first-order algorithm from \cite{vural2022mirror} called stochastic mirror descent algorithm  
with  uniformly convex mirror map (\texttt{RSMD}). It is based on  stochastic mirror descent algorithm  \cite{nemirovskij1983problem} and the notion of uniform convexity (to be determined further).}
%
%
\begin{definition}[Uniform convexity] Consider a differentiable convex function $\psi : \RR^d \rightarrow \RR$, an exponent $r \geq 2$, and a constant $K > 0$. Then, $\psi$ is called $(K,r)$-uniformly convex w.r.t. the  $\ell_p$-norm if, for any $x,y \in \RR^d$,
\begin{equation}\label{def: uniform convex}
\psi(y) - \psi(x)- \langle  \nabla \psi(x), y -x \rangle \geq \frac{K}{r}\|x-y\|^r_p.
\end{equation}
\end{definition}
When $r=2$ the definition of $(K,r)$-uniform convexity is equivalent to $K$-strongly convexity. Examples of functions when $r>2$ can be obtained from the next lemma. 

\begin{lemma}\label{func}
For $\kappa\in (0,1], q \in [1 + \kappa, \infty)$ and $p$ such that $ \nicefrac{1}{q} + \nicefrac{1}{p} = 1$, we define 
\begin{equation}\label{eq: K_q def}
    K_q  \eqdef 10\max\left\{ 1, (q-1)^{\frac{1+\kappa}{2}}\right\}.
\end{equation}
Then,
\begin{equation}\label{eq: phi_p def}
    \phi_p(x) \eqdef \frac{\kappa}{1 + \kappa} \|x\|_{p}^{\frac{1+\kappa}{\kappa}}
\end{equation}
is $\left( K_q^{-\frac1\kappa}, \frac{1+\kappa}{\kappa} \right)$-uniformly convex w.r.t. the $\ell_p$-norm.

\end{lemma}

Now we describe robust stochastic mirror descent (RSMD) algorithm \cite{vural2022mirror}. Let function $\Psi: \RR^d \rightarrow \RR$  be $(K,r)$-uniformly convex w.r.t. the $\ell_p$-norm.
We denote its Fenchel conjugate and its Bregman divergence respectively as
$$\Psi^*(y)= \sup\limits_{x \in \RR^d} \{\langle x,y \rangle - \Psi(x) \} \quad \text{and} \quad  D_{\Psi}(y,x) = \Psi(y) - \Psi(x) - \langle \nabla \Psi(x), y -x \rangle.$$
For a given stepsize $\nu$  and gradient $g_{k+1}$, the
updates of RSMD are defined as follows:
\begin{equation}\label{MD}
y_{k+1} = \nabla(\Psi^*) (\nabla \Psi(x_k) - \nu g_{k+1}), \quad x_{k+1} = \arg \min\limits_{x \in \mathcal{X}} D_{\Psi}(x, y_{k+1}).
\end{equation}
Using the assumptions on the function $\Psi$, it can be proved that the updates are well-defined and that $(\nabla \Psi)^{-1} = \nabla \Psi^*$. The map $\nabla \Psi$ \dd{is referred to as the
mirror map.}
The next theorem presents the convergence guarantee for the RSMD. Let \[x^* = \arg\min \limits_{x \in \mathcal{X}} f(x).\]


\begin{theorem} \cite[Theorem $6$]{vural2022mirror}\label{MD convergence}
Consider some $\kappa \in (0,1] , p \in [1, \infty]$ and prox-function $\Psi_p$ which is $\left( 1, \frac{1+\kappa}{\kappa}\right)$-uniformly convex w.r.t. $p$ norm. Then, for the SMD Algorithm outlined in  \eqref{MD}, after $T$ iterations with any $g_k \in \RR^d, k \in \overline{1, T}$ and starting point $x_0 =\arg\min\limits_{x \in \mathcal{X}} \Psi_p(x)$ we have
\begin{equation}\label{eq: MD conv}\frac1T \sum \limits_{k=0}^{T-1} \langle g_{k+1}, x_k  -x^* \rangle \leq \frac{\kappa}{\kappa +1} \frac{R_{0}^{\frac{1+\kappa}{\kappa}} }{\nu T} + \frac{\nu^{\kappa}}{1+\kappa} \frac{1}{T} \sum \limits_{k=0}^{T-1} \|g_{k+1}\|^{1+\kappa}_q,\end{equation} 
where $R_{0}^{\frac{1+\kappa}{\kappa}} \eqdef  \frac{1+\kappa}{\kappa}  D_{\Psi_{p}}(x^*, x_0)$ is the distance between starting point $x_0$ and solution $x^*$.
\end{theorem}

\subsection{ Zeroth order version of \texttt{RSMD} (\texttt{ZO-RSMD})}
\dm{Next, we present our first zeroth-order algorithm called \texttt{ZO-RSMD} (zeroth-order version of robust SMD algorithm). }
The main idea of the proposed \texttt{ZO-RSMD} (zeroth-order version of robust SMD algorithm) is to combine the above \texttt{RSMD} algorithm  \eqref{MD} with the two-point gradient approximation \eqref{g}. 

\begin{algorithm}[h!]
\caption{ \dm{\texttt{ZO-RSMD} }}\label{alg:Robust}
\begin{algorithmic}[1]
\Procedure{\texttt{ZO-RSMD} }{number of iterations $T$, stepsize $\nu$, \dd{prox-}function $\Psi_{p}$, smoothing constant $\tau$}
    \State $x_0 \gets \arg\min\limits_{x \in \mathcal{X}} \Psi_{p}(x)$
    
    \For{$k = 0, 1, \dots ,T-1$}  
    
        \State Sample $\mathbf{e}_k \sim {\rm Uniform}(\{\mathbf{e}: \|\mathbf{e}\|_2 = 1 \}) $ independently 
        
        \State Sample $\xi_k$ independently
        
        \State Calculate $g_{k+1} =  \frac{d}{2\tau}(\phi(x_k + \tau \mathbf{e}_k, \xi_k) - \phi(x_k - \tau \mathbf{e}_k, \xi_k)) \mathbf{e}_k$
        
        \State Calculate $y_{k+1}  \gets \nabla(\Psi_{p}^*) (\nabla \Psi_{p}(x_k) - \nu g_{k+1})$
        
        \State Calculate $x_{k+1}  \gets \arg \min\limits_{x \in \mathcal{X}} D_{\Psi_{p}}(x, y_{k+1})$

    \EndFor

    \State \textbf{return} $\overline{x}_T \gets \frac{1}{T}\sum\limits_{k=0}^{T-1}x_k$
 \EndProcedure
\end{algorithmic}
\end{algorithm}

The next theorem provides  a convergence guarantee for \texttt{ZO-RSMD} (see Algorithm \ref{alg:Robust}). 

\begin{theorem}\label{Robust Conv} Let function $f$ satisfy Assumptions \ref{as: convex}, \ref{as: Lipshcitz and bounded}, \ref{as: noize},  $q \in [1 + \kappa, \infty]$. Let 
$\Psi_{p}(x)$ be a prox-function which is  $\left( 1, \frac{1+\kappa}{\kappa}\right)$-uniformly convex w.r.t. the $\ell_p$-norm   (e.g., $\Psi_{p}(x) = K_{q}^{1/\kappa} \phi_p(x)$, where $K_{q}$, $\phi_p$ are defined in \eqref{eq: K_q def} and \eqref{eq: phi_p def} respectively). Let   stepsize $\nu = \frac{R_{0}^{1/\kappa}}{\sigma_{q} }T^{-\frac{1}{1+\kappa}}$ with $\sigma_{q}$ given in Lemma \ref{lem: grad norm 1 +k intro}, distance between starting point $x_0$ and solution $x^*$ $R_{0}^{\frac{1+\kappa}{\kappa}} \eqdef  \frac{1+\kappa}{\kappa}  D_{\Psi_{p}}(x^*, x_0)$ and diameter $\mathcal{D}_{\Psi}^\frac{1+\kappa}{\kappa} \eqdef \frac{1+\kappa}{\kappa} \sup\limits_{x,y \in \mathcal{X}}  D_{\Psi_{p}}(x,y) .$
Then for the output $\overline{x}_T$ of the  Algorithm \ref{alg:Robust}  the following holds 
\begin{enumerate} 
    \item 
\begin{equation}\label{eq:conv convergence}
\EE[f(\overline{x}_T)] - f(x^*) \leq 2M_2\tau + \frac{\sqrt{d}\Delta}{\tau}\mathcal{D}_{\Psi}   + \frac{R_0 \sigma_{q} }{T^\frac{\kappa}{1+\kappa}},\end{equation}

where  $\sigma_{q}^{1+\kappa} = 2^{\kappa}\left(\frac{\sqrt{d}}{2^{1/4}} a_{q}M_2\right)^{1+\kappa} + 2^{\kappa}\left(\frac{da_{q}\Delta}{\tau}\right)^{1+\kappa}$.
\item Moreover, with optimal $\tau =  \sqrt{\frac{\sqrt{d}\Delta\mathcal{D}_{\Psi} + 4R_{0}da_{q}\Delta T^{-\frac{\kappa}{1+\kappa}}}{2M_2}}$, we have
\begin{equation}\label{eq: robust disc deter}
\EE[f(\overline{x}_T)] - f(x^*) \leq  \sqrt{8M_2\sqrt{d}\Delta\mathcal{D}_\Psi} + \sqrt{\frac{32M_2R_{0}da_{q}\Delta }{T^{\frac{\kappa}{(1+\kappa)}}}}   + \frac{2\sqrt{d} a_{q}M_2 R_0}{T^{\frac{\kappa}{1+\kappa}}}.\end{equation}

\end{enumerate}
\end{theorem}

\begin{proof}[Sketch of the Proof of Theorem \ref{Robust Conv} ]

the proof is based on Theorem \ref{MD convergence} and inequality \eqref{eq: MD conv} which give
\begin{equation}\label{eq: main inq in robust}
    \underbrace{\EE \left[\frac1T \sum \limits_{k=0}^{T-1} \langle g_{k+1}, x_k  -x^* \rangle \right]}_\text{\circled{1}} \leq \underbrace{\EE \left[ \frac{\kappa}{\kappa +1} \frac{R_{0}^{\frac{1+\kappa}{\kappa}} }{\nu T}\right]}_\text{\circled{2}} + \underbrace{\EE \left[ \frac{\nu^{\kappa}}{1+\kappa} \frac{1}{T} \sum \limits_{k=0}^{T-1} \|g_{k+1}\|^{1+\kappa}_q\right]}_\text{\circled{3}}.\end{equation}
\circled{1} term in \eqref{eq: main inq in robust} due to convexity and approximation properties of $\hat{f}_\tau(x)$ in Lemma \ref{lem: hat_f properties intro} and measure concentration Lemma \ref{inner product grad r} can be bounded as
$$\circled{1} \geq \EE[f(\overline{x}_T)] - f(x^*) - 2M_2\tau -  \frac{\sqrt{d}\Delta}{\tau}\mathcal{D}_{\Psi}.$$
\circled{3} term in \eqref{eq: main inq in robust} can be bounded by Lemma \ref{lem: grad norm 1 +k intro} as
$$\circled{3} \leq \frac{\nu^{\kappa}}{1+\kappa} \sigma_q^{1+\kappa}.$$
Combining these bounds together, we get
$$\EE[f(\overline{x}_T)] - f(x^*) \leq 2M_2\tau + \frac{\sqrt{d}\Delta}{\tau}\mathcal{D}_\Psi+ \frac{R_{0}^{\frac{1+\kappa}{\kappa}} }{\nu T} + \frac{\nu^{\kappa}}{1+\kappa}\sigma_{q}^{1+\kappa}.$$
Next we choose optimal stepsize $\nu = \frac{R_{0}^{1/\kappa}}{\sigma_{q} }T^{-\frac{1}{1+\kappa}}, \tau$  and finish the proof. 
\end{proof}
For the complete proof we refer to  Section \ref{sec: proof robust conv}.

\subsection{Discussion}

\paragraph{Maximum admissible level of adversarial noise}

Let $\varepsilon>0 $ be a desired accuracy in terms of the function value, i.e., our goal is to guarantee $\EE[f(\overline{x}_T)] - f(x^*) \leq \varepsilon$.
According to Theorem \ref{Robust Conv} in the case of absence of the adversarial noise, i.e., when $\Delta = 0$, the iteration complexity to reach accuracy $\varepsilon$ is $T = \dm{O}  \left(\left(\frac{R_{0}\sqrt{d}a_qM_2}{\varepsilon}\right)^\frac{1+\kappa}{\kappa}\right)$ if $\tau $ is chosen sufficiently small. This complexity is optimal according to \cite{nemirovskij1983problem} in terms of $\varepsilon$ dependency. 
In order to obtain the same complexity in the case when $\Delta > 0$, we need to choose an appropriate value of $\tau$ and ensure that $\Delta$ is sufficiently small. Thus, the terms $2M_2\tau$ and $\frac{\sqrt{d}\Delta}{\tau}\mathcal{D}_{\Psi}$ in \eqref{eq:conv convergence} should be of the order $ \varepsilon$. These conditions also make negligible the $\tau$-depending term in $\sigma_{q}$. One can choose $\tau = \frac{\varepsilon}{M_2}$ rather than optimal $\tau$ proposed in Theorem \ref{Robust Conv} in order to get easier calculations.   Consequently, when $\tau = \frac{\varepsilon}{M_2} $ and $ \Delta \leq  \frac{\varepsilon^2}{M_2\sqrt{d}\mathcal{D}_{\Psi}}$, we have
$$    T = \dm{O}\left(\left(\frac{R_{0}\sqrt{d}a_qM_2}{\varepsilon}\right)^\frac{1+\kappa}{\kappa}\right).$$
According to \cite{pasechnyuk2023upper, risteski2016algorithms} bound $ \Delta \leq  \frac{\varepsilon^2}{M_2\sqrt{d}\mathcal{D}_{\Psi}}$ exactly matches the upper bound of admissible adversarial noise for non-smooth zeroth-order optimization.

\paragraph{Dependency of the bounds on $q$ and $d$}

In Algorithm \ref{alg:Robust}, we can freely choose $p\in [1,2]$ and $\Psi_p$, which  lead to different values of $\mathcal{D}_\Psi, R_{0}, a_q$ depending on the compact convex set $\mathcal{X}$. It is desirable to reduce $a_q, \mathcal{D}_\Psi$ simultaneously, that would allow us to increase maximal noise level $\Delta$ and converge faster without changing the rate according to \eqref{eq:conv convergence}. Yet, unlike the well-studied SMD Algorithm \cite{nemirovskij1983problem}  with strongly convex prox-functions $\Psi_p$, there are only a few examples of effective choices of uniformly-convex prox-functions $\Psi_p$.

\section{ \dm{ Second Algorithm: \texttt{ZO-Clip-SMD}}}\label{sec:ZOCLipping}
\dm{In this section, we present our second algorithm which 
 is based on the mirror descent and gradient clipping technique. }

An alternative approach for dealing with heavy-tailed noise distributions in stochastic optimization is based on the gradient clipping technique, see e.g., \cite{zhang2020adaptive}.  Given a constant $c>0$, the clipping operator applied to a vector $g$ is given by
$$\hat{g} = \begin{cases}\frac{g}{\|g\|} \min(\|g\|, c), & g \neq 0, \\ 0, &g = 0. \end{cases}$$ 
Clipped gradient has several useful properties for further proofs. 
\begin{lemma}\label{clip_grad_properties intro}
For $c >0$ and stochastic vector $g = g(x,\xi,\mathbf{e})$ we define $\hat{g} = \frac{g}{\|g\|_q} \min(\|g\|_q, c)$. Then we have
\begin{enumerate}
    \item 
    \begin{equation}\label{clip grad  EE diff}\|\hat{g} - \EE[\hat{g} ]\|_q \leq 2c.\end{equation}
    
    \item 
     Also if \hspace{1pt}  $\EE[\|g(x,\xi,\mathbf{e})\|_q^{1+\kappa}]\leq  \sigma_{q}^{1+\kappa}$, then we have\begin{enumerate}
    \item \begin{equation}\label{clip grad  second moment}\EE[\|\hat{g}\|^2_q] \leq  \sigma_{q}^{1+\kappa} c^{1 - \kappa},\end{equation}
    \item \begin{equation}\label{clip grad  var}\EE[\|\hat{g} - \EE[\hat{g}]\|^2_q] \leq  4\sigma_{q}^{1+\kappa} c^{1 - \kappa},\end{equation}
    \item \begin{equation}\label{clip grad nabla f tau diff}\|\EE[g] - \EE [\hat{g}]\|_q  \leq \frac{\sigma_{q}^{1+\kappa}}{c^\kappa}.\end{equation}
    
\end{enumerate}
\end{enumerate}
\end{lemma}

The clipping constant $c$ allows playing with the trade-off between the faster convergence due to bounded second moment of $\hat{g}$ and bias $\|\EE[\hat{g} - g]\|$ when $c \rightarrow 0$. 

\begin{algorithm}
\caption{ \dm{ \texttt{ZO-Clip-SMD} }  }\label{alg:clip}
\begin{algorithmic}[1]
\Procedure{\dm{ \texttt{ZO-Clip-SMD} } }{Number of iterations $T$, stepsize $\nu$, clipping constant $c$, \dd{prox-}function $\Psi_{p}$, smoothing constant $\tau$}
    \State $x_0 \gets \arg\min\limits_{x \in \mathcal{X}} \Psi_{p}(x)$
    
    \For{$k = 0, 1, \dots ,T-1$}  
    
        \State Sample $\mathbf{e}_k \sim {\rm Uniform}(\{\mathbf{e}: \|\mathbf{e}\|_2 = 1 \}) $ independently 
        
        \State Sample $\xi_k$ independently
        
        \State Calculate $g_{k+1} = \frac{d}{2\tau}(\phi(x_k + \tau \mathbf{e}_k, \xi_k) - \phi(x_k - \tau \mathbf{e}_k, \xi_k)) \mathbf{e}_k$
         
        \State Calculate $\hat{g}_{k+1} = \frac{g_{k+1}}{\|g_{k+1}\|_q} \min(\|g_{k+1}\|_q, c)$ 
        
        \State Calculate $y_{k+1}  \gets \nabla(\Psi_{p}^*) (\nabla \Psi_{p}(x_k) - \nu \hat{g}_{k+1})$
        
        \State Calculate $x_{k+1}  \gets \arg \min\limits_{x \in \mathcal{X}} D_{\Psi_{p}}(x, y_{k+1})$

    \EndFor

    \State \textbf{return} $\overline{x}_T \gets \frac{1}{T}\sum\limits_{k=0}^{T-1}x_k$
\EndProcedure
\end{algorithmic}
\end{algorithm}
The next theorem presents  convergence rates for \dm{ \texttt{ZO-Clip-SMD}} (Algorithm \ref{alg:clip}) in terms of the expectation of the suboptimality gap.
\begin{theorem}\label{Clip Conv} Let function $f$ satisfy  Assumptions \ref{as: convex}, \ref{as: Lipshcitz and bounded}, \ref{as: noize},  $q \in [2, \infty]$, arbitrary number of iterations $T$, smoothing constant $\tau > 0$  be given. Let 
$\Psi_{p}(x)$ be a prox-function  which is
$1$-strongly convex w.r.t. the $p$-norm . Let the stepsize $\nu = \left( \frac{R_{0}^2}{4T\sigma_{q}^{1+\kappa} \mathcal{D}_\Psi^{1-\kappa} }\right)^{\frac{1}{1+\kappa}}$ with $\sigma_{q}$ given in Lemma \ref{lem: grad norm 1 +k intro}, distance between starting point $x_0$ and solution $x^*$ $R_{0}^{\frac{1+\kappa}{\kappa}} \eqdef  \frac{1+\kappa}{\kappa}  D_{\Psi_{p}}(x^*, x_0)$, diameter $\mathcal{D}_{\Psi}^\frac{1+\kappa}{\kappa} \eqdef \frac{1+\kappa}{\kappa} \sup\limits_{x,y \in \mathcal{X}}  D_{\Psi_{p}}(x,y),$ and    the clipping constant $c =  \frac{2\kappa\mathcal{D}_\Psi}{(1-\kappa)\nu}$.
Then for the output  $\overline{x}_T$ of Algorithm \ref{alg:clip} the following holds  
\begin{enumerate}
    \item  Then, we have
\begin{equation}\label{eq: clip theorem eq}  
\EE[f(\overline{x}_T)] - f(x^*) = 2M_2\tau + \frac{\sqrt{d}\Delta}{\tau}\mathcal{D}_{\Psi}   + \frac{R_{0}^\frac{2\kappa}{1+\kappa} \mathcal{D}_\Psi^{\frac{1-\kappa}{1+\kappa}}\sigma_{q} }{T^\frac{\kappa}{1+\kappa}},\end{equation}
where  $\sigma_{q}^{1+\kappa} = 2^{\kappa}\left(\frac{\sqrt{d}}{2^{1/4}} a_{q}M_2\right)^{1+\kappa} + 2^{\kappa}\left(\frac{da_{q}\Delta}{\tau}\right)^{1+\kappa}$.

\item Moreover, with the optimal $\tau =  \sqrt{\frac{\sqrt{d}\Delta\mathcal{D}_{\Psi} + 4R_{0}^\frac{2\kappa}{1+\kappa} \mathcal{D}_\Psi^{\frac{1-\kappa}{1+\kappa}}da_{q}\Delta T^{-\frac{\kappa}{1+\kappa}}}{2M_2}}$, we have
\begin{eqnarray}
\EE[f(\overline{x}_T)] - f(x^*) &\leq&  \sqrt{8M_2\sqrt{d}\Delta\mathcal{D}_\Psi} + \sqrt{\frac{32M_2R_{0}^\frac{2\kappa}{1+\kappa} \mathcal{D}_\Psi^{\frac{1-\kappa}{1+\kappa}}da_{q}\Delta }{T^{\frac{\kappa}{(1+\kappa)}}}} \notag \\  
&+& \frac{2\sqrt{d} a_{q}M_2 R_{0}^\frac{2\kappa}{1+\kappa} \mathcal{D}_\Psi^{\frac{1-\kappa}{1+\kappa}}} {T^{\frac{\kappa}{1+\kappa}}}\label{eq: clip disc deter}.\end{eqnarray}
\end{enumerate}
\end{theorem}
\begin{proof}[Sketch of the Proof of Theorem \ref{Clip Conv} ]
The proof is based on Theorem \ref{MD convergence} and inequality \eqref{eq: MD conv} for $1$-strongly convex $\Psi_p$, which give
\begin{equation}\label{eq: main inq in clip}
    \underbrace{\EE \left[\frac1T \sum \limits_{k=0}^{T-1} \langle \hat{g}_{k+1}, x_k  -x^* \rangle  \right]}_\text{\circled{1}} \leq  \EE \left[\frac{1}{2} \frac{R_{0}^{2} }{\nu T} \right] + \underbrace{\EE \left[\frac{\nu}{2} \frac{1}{T} \sum \limits_{k=0}^{T-1} \|\hat{g}_{k+1}\|^{2}_q \right]}_\text{\circled{2}}.\end{equation}
\circled{1} term in \eqref{eq: main inq in clip} due to convexity and approximation properties of $\hat{f}_\tau(x)$ in Lemma \ref{lem: hat_f properties intro}, measure concentration Lemma \ref{inner product grad r} and clipping properties in Lemma \ref{clip_grad_properties intro} can be bounded as
$$\circled{1} \geq \EE[f(\overline{x}_T)] - f(x^*) - 2M_2\tau -  \frac{\sqrt{d}\Delta}{\tau}\mathcal{D}_{\Psi} - \frac{\mathcal{D}_{\Psi} \sigma_q^{1+\kappa}}{c^\kappa}.$$
\circled{2} term in \eqref{eq: main inq in clip} can be bounded by Lemma \ref{clip_grad_properties intro} as
$$\circled{2} \leq \frac{\nu}{2} c^{1-\kappa} \sigma_q^{1+\kappa}.$$ 
Combining these bounds together, we get
$$\EE [f(\overline{x}_T)] - f(x^*)  \leq 2M_2\tau + \frac12 \frac{R_{0}^{2} }{\nu T} + \frac{\nu}{2}  \sigma_{q}^{1+\kappa} c^{1 - \kappa} + \left(\frac{\sigma_{q}^{1+\kappa}}{c^\kappa} + \Delta \frac{\sqrt{d}}{\tau}\right) \mathcal{D}_\Psi. $$
Next, we choose optimal clipping constant $c  =  \frac{2\kappa\mathcal{D}_\Psi}{(1-\kappa)\nu}$. Then, the optimal stepsize $\nu  = \left( \frac{R_{0}^2}{4T\sigma_{q}^{1+\kappa} \mathcal{D}_\Psi^{1-\kappa} }\right)^{\frac{1}{1+\kappa}}$ and smoothing parameter $\tau$ finish the proof. 
\end{proof}
For the complete proof we refer to \dm{Appendix} \ref{sec: clip conv}.

\dm{The next theorem present the convergence rates of   
for \dm{ \texttt{ZO-Clip-SMD}} (Algorithm \ref{alg:clip}) 
with high probability rather than in expectation.}
 We will use the $\widetilde{O}(\cdot)$-notation to hide polynomial factors of $\log\frac1\delta$. 
\begin{theorem}\label{Clip Conv prob} Let function $f$ satisfy Assumptions \ref{as: convex}, \ref{as: Lipshcitz and bounded}, \ref{as: noize},  $q \in [2, \infty]$, arbitrary number of iterations $T$, smoothing constant $\tau > 0$  be given. Let
$\Psi_{p}(x)$ be a
$1$-strongly convex w.r.t. the $p$-norm prox-function. Let the clipping constant $c = T^{\frac{1}{(1+\kappa)}}\sigma_{q}$ with $\sigma_{q}$ given in Lemma \ref{lem: grad norm 1 +k intro}, the stepsize $\nu = \frac{\mathcal{D}_{\Psi}}{c}$ with diameter $\mathcal{D}_{\Psi}^2 \eqdef 2 \sup\limits_{x,y \in \mathcal{X}}  D_{\Psi_{p}}(x,y) .$
Then for the output  $\overline{x}_T$ of the Algorithm \ref{alg:clip} the following holds
\begin{enumerate}
    \item 
Then, with probability at least $1 - \delta$, we have

\begin{equation}\label{eq: clip prob theorem eq}f(\overline{x}_T) - f(x^*) \leq 2M_2\tau  +  \frac{\Delta\sqrt{d}}{\tau}\mathcal{D}_{\Psi} + \widetilde{O}\left(\frac {\mathcal{D}_{\Psi}\sigma_{q}}{T^{\frac{\kappa }{1+\kappa}}} \right),  \end{equation}

where  $\sigma_{q}^{1+\kappa} = 2^{\kappa}\left(\frac{\sqrt{d}}{2^{1/4}} a_{q}M_2\right)^{1+\kappa} + 2^{\kappa}\left(\frac{da_{q}\Delta}{\tau}\right)^{1+\kappa}$.
\item Moreover, with the optimal $\tau =  \sqrt{\frac{\sqrt{d}\Delta\mathcal{D}_{\Psi} + 4\mathcal{D}_{\Psi}da_{q}\Delta T^{-\frac{\kappa}{1+\kappa}}}{2M_2}}$, we have
\begin{equation}\label{eq: clip prob disc deter}
f(\overline{x}_T) - f(x^*) =\widetilde{O}\left(  \sqrt{8M_2\sqrt{d}\Delta\mathcal{D}_\Psi} + \sqrt{\frac{32M_2\mathcal{D}_{\Psi}da_{q}\Delta }{T^{\frac{\kappa}{(1+\kappa)}}}}   + \frac{2\sqrt{d} a_{q}M_2 \mathcal{D}_{\Psi} }{T^{\frac{\kappa}{1+\kappa}}}\right).\end{equation}

\end{enumerate}
\end{theorem}

\begin{proof}[Sketch of the Proof of Theorem \ref{Clip Conv prob} ]

To bound variables with probability at least $1-\delta$ we use the classical Bernstein inequality for the sum of martingale differences (i.e. $\EE[X_i| X_{j< i}] = 0,$ for all $  i \geq 1$)  (Lemma \ref{Bernstein inequality}) and the sum of squares of random variables (Lemma \ref{Sum of squares MSD}).

The proof is based on Theorem \ref{MD convergence} and inequality \eqref{eq: MD conv} for $1$-strongly convex $\Psi_p$ which give
\begin{equation}\label{eq: main inq in clip prob}
    \frac1T \sum \limits_{k=0}^{T-1} \langle \hat{g}_{k+1}, x_k  -x^* \rangle \leq  \frac{1}{2} \frac{R_{0}^{2} }{\nu T}  + \underbrace{\frac{\nu}{2} \frac{1}{T} \sum \limits_{k=0}^{T-1} \|\hat{g}_{k+1}\|^{2}_q }_\text{\circled{1}}.\end{equation}

    Adding $\pm \EE_{|\leq k}[\hat{g}_{k+1}] $ and $\pm \hat{f}_\tau(x_k)$ to the left part of \eqref{eq: main inq in clip prob}, we obtain
    \begin{align*}
    \frac1T \sum \limits_{k=0}^{T-1}\langle \hat{g}_{k+1}, x_k  -x^* \rangle 
    &=\underbrace{\frac1T \sum \limits_{k=0}^{T-1}\langle \hat{g}_{k+1} - \EE_{|\leq k}[\hat{g}_{k+1}]  , x_k  -x^* \rangle}_\text{\circled{2}} \\
    &+ \underbrace{\frac1T \sum \limits_{k=0}^{T-1}\langle \EE_{|\leq k}[\hat{g}_{k+1}] - \nabla \hat{f}_\tau(x_k), x_k  -x^* \rangle}_\text{\circled{3}},     \\
    &+ \underbrace{\frac1T \sum \limits_{k=0}^{T-1}\langle \nabla \hat{f}_\tau(x_k), x_k  -x^* \rangle}_\text{\circled{4}}.
    \end{align*}
We bound \circled{1} term in \eqref{eq: main inq in clip prob} using Lemma \ref{Sum of squares MSD} and \circled{2} as the sum of martingale differences using Lemma~\ref{Bernstein inequality}:
$$\circled{1}= \widetilde{O}\left(\sigma_{q,\kappa}^{1+\kappa} c^{1 - \kappa}  + \frac{1}{T}c^2\right)$$
$$\circled{2} = \widetilde{O}\left(\frac{4c \mathcal{D}_\Psi}{T} + \frac{\sqrt{4\sigma_{q}^{1+\kappa} c^{1 - \kappa}}}{\sqrt{T}} \mathcal{D}_\Psi^2  \right).$$
Next, we bound  \circled{4} using the convexity of $\hat{f}_\tau(x)$ in Lemma \ref{lem: hat_f properties intro} and \circled{3} using the measure concentration Lemma \ref{inner product grad r} and clipping properties in Lemma \ref{clip_grad_properties intro}: 
$$\circled{3} \leq\left( \frac{\sigma_{q}^{1+\kappa}}{c^\kappa} + \Delta\frac{\sqrt{d}}{\tau}\right)\mathcal{D}_\Psi,$$
$$\circled{4} \geq   f(\overline{x}_T) - f(x^*) -2M_2\tau.$$
Combining these bounds together, we get
\begin{align*}
    f(\overline{x}_T) - f(x^*) 
    &\leq 2M_2\tau + \left( \frac{\sigma_{q}^{1+\kappa}}{c^\kappa} + \Delta\frac{\sqrt{d}}{\tau}\right)\mathcal{D}_\Psi+\frac12  \frac{R_{0}^{2} }{\nu T} \\
    &+ \widetilde{O}\left(\frac{\nu}{2} \sigma_{q}^{1+\kappa} c^{1 - \kappa}  + \frac{\nu}{2}\frac{1}{T}c^2 +  \frac{4c\mathcal{D}_\Psi}{T} + \frac{\sqrt{4\sigma_{q}^{1+\kappa} c^{1 - \kappa}}}{\sqrt{T}} \mathcal{D}_\Psi^2\right).
\end{align*}
Next, we choose the stepsize $ \nu = \frac{\mathcal{D}_{\Psi}}{c}$ ,  clipping constant $c = T^{\frac{1}{(1+\kappa)}}\sigma_{q}$, smoothing parameter $\tau$, and finish the proof. 
\end{proof}
For the complete proof we refer to  Section \ref{sec: clip conv prob}.

\subsection{Discussion}

\paragraph{Maximum  admissible level of adversarial noise}

Let $\varepsilon >0$ be a desired accuracy in terms of the function value, i.e., with probability at least $1-\delta$ we have  $f(\overline{x}_T) - f(x^*) \leq \varepsilon$.
In Theorem \ref{Clip Conv prob} if there is no adversarial noise, i.e., $\Delta = 0$, then the \dd{number of iterations} $T$ to reach this accuracy is given by $T = \widetilde{O}\left(\left(\frac{\mathcal{D}_{\Psi}\sqrt{d}a_qM_2}{\varepsilon}\right)^\frac{1 + \kappa}{\kappa}\right)$  when $\tau \rightarrow 0$. This bound is optimal in terms of $\varepsilon$ dependency according to \cite{nemirovskij1983problem}. 
In order to keep the same complexity when $\Delta > 0$, the terms $2M_2\tau$ and $\frac{\sqrt{d}\Delta}{\tau}\mathcal{D}_{\Psi}$ should be of the order $ \varepsilon$. These conditions also make negligible the $\tau$-depending term in $\sigma_{q}$. One can choose $\tau = \frac{\varepsilon}{M_2}$ rather than optimal $\tau$ proposed in Theorem \ref{Clip Conv prob} in order to get easier calculations.  Consequently, if 
$ \quad \tau = \frac{\varepsilon}{M_2}$ and $\Delta \leq  \frac{\varepsilon^2}{M_2\sqrt{d}\mathcal{D}_{\Psi}} $
then 
$$   T = \widetilde{O}\left(\left(\frac{\mathcal{D}_\Psi\sqrt{d}a_qM_2}{\varepsilon}\right)^\frac{1 + \kappa}{\kappa}\right).$$
According to \cite{pasechnyuk2023upper, risteski2016algorithms} bound $ \Delta \leq  \frac{\varepsilon^2}{M_2\sqrt{d}\mathcal{D}_{\Psi}}$ exactly matches the upper bound of admissible adversarial noise for non-smooth zeroth-order optimization.


\paragraph{Recommendations for choosing $\Psi_p$}

In Algorithm \ref{alg:clip}, we can freely choose $p\in [1,2]$ and $\Psi_p$, which, depending on the compact convex set $\mathcal{X}$, will change $\mathcal{D}_\Psi, R_{0}, a_q$. The main task is to reduce $a_q, \mathcal{D}_\Psi$ simultaneously, which will allow us to increase maximal noise $\Delta$ and converge faster without changing the rate according to \eqref{eq: clip prob theorem eq}. 

Next, we discuss some standard sets $\mathcal{X}$ and prox-functions $\Psi_p$ taken from \cite{ben2001lectures}. The two main setups are given by
\begin{enumerate}
    \item Ball setup:
    \begin{equation}\label{eq: ball setup} \qquad  p = 2, \Psi_p(x) = \frac{1}{2}\|x\|_2^2,
    \end{equation}
    \item Entropy setup:
\begin{equation}\label{eq: entropy setup}  \qquad  p = 1, \Psi_p(x) = (1+\gamma)\sum_{i=1}^d (x_i + \gamma/d)\log(x_i + \gamma/d), \gamma > 0.
    \end{equation}
\end{enumerate}
We consider unit balls $B^d_{p'}$  and standard simplex $\Delta^d_+  = \{x \in \RR^d: x \geq 0, \sum_i x_i  = 1 \}$ as $\mathcal{X}$. By Lemma \ref{lem: grad norm 1 +k intro} constant $a_q$ equals $d^{\frac1q - \frac12}  \min \{ \sqrt{32\ln d - 8} , \sqrt{2q - 1}\}.$ 
The next tables collect the iteration complexity $T^{\frac{\kappa}{1+\kappa}}$ and maximum feasible noise level $\Delta$ up to $O\left(\log\frac1\delta\right)$ factor for each setup (row) and set (column).
\begin{table}[h!]
\centering
 \caption{$T^{\frac{\kappa}{1+\kappa}} $ up to $O\left(\log\frac1\delta\right)$ factor for Algorithm \ref{alg:clip} }
\begin{tabular}{ |c|c|c|c|c| } 

 \hline
  & $\Delta^d_+$ & $B^d_1$ & $B^d_2$ & $B^d_\infty$ \\ 
  \hline
  Ball & $\sqrt{d}M_2/\varepsilon$&$\sqrt{d}M_2/\varepsilon$&$\sqrt{d}M_2/\varepsilon$&$dM_2/\varepsilon$\\
 \hline
 Entropy & $\ln d M_2 /\varepsilon$ & $\ln d M_2 /\varepsilon$ & $\sqrt{d}\ln d M_2 /\varepsilon$ & $d\ln d M_2 /\varepsilon$\\
 \hline
 \end{tabular}

 \end{table}

 \begin{table}[h!]\label{table: delta}
\centering
 \caption{Maximum feasible noise level $\Delta $ up to $O\left(1\right)$ factor for Algorithm \ref{alg:clip} }
\begin{tabular}{ |c|c|c|c|c| } 

 \hline
  & $\Delta^d_+$ & $B^d_1$ & $B^d_2$ & $B^d_\infty$ \\ 
  \hline
  Ball & $\varepsilon^2 / (\sqrt{d}M_2)$&$\varepsilon^2 / (\sqrt{d}M_2)$&$\varepsilon^2 / (\sqrt{d}M_2)$&$\varepsilon^2 / (d M_2)$\\
 \hline
 Entropy & $\varepsilon^2 / (\sqrt{d\ln d}M_2)$ & $\varepsilon^2 / (\sqrt{d\ln d}M_2)$ & $\varepsilon^2 / (d\sqrt{\ln d}M_2)$ & $\varepsilon^2 / (\sqrt{d^3\ln d}M_2)$\\
 \hline
 \end{tabular}

 \end{table}
 
From these tables, we see that for $\mathcal{X} = \Delta_+^d$ or $ B^d_1$, the Entropy setup is preferable, while the Ball setup allows maximum feasible noise level $\Delta$ to be up to $\sqrt{\ln d}$  greater. Meanwhile, for $\mathcal{X} = B^d_2$ or $ B^d_\infty$, the Ball setup is better in terms of both convergence rate and noise robustness.

\paragraph{\dm{Comparison of two algorithms:  \texttt{ZO-RSMD} and  \texttt{ZO-Clip-SMD} }}
\dm{Despite the fact that both algorithms have the same convergence rates, \texttt{ZO-Clip-SMD} is  more flexible due to the greater freedom of choice of prox-functions $\Psi_p$.
However, its convergence  dramatically depends on the clipping constant $c$ which must be carefully chosen. 
}


\section{ \dm{Algorithms with Restarts:  \texttt{ZO-Restarts} }}
\label{sec:restarts}
In this section, we assume the objective function satisfies the $r$-growth condition \cite{shapiro2021lectures}. In this case, optimization algorithms can be accelerated by using the restart technique \cite{juditsky2014deterministic}.
\begin{assumption}\label{r-growth}
Function $f$ is $r$-growth  function if there are $r \geq 1$ and $\mu_r \geq 0$ such that for all $x$
$$\frac{\mu_r}{2}\|x-x^*\|_{p}^r \leq f(x) - f(x^*),$$
where $x^*$ is problem solution.
\end{assumption}
In particular, the condition of $\mu$-strong convexity w.r.t. the $\ell_p$-norm is the $2$-growth condition. The restart technique works if $\Delta$ is small enough to keep the optimality of Algorithms \ref{alg:Robust} and \ref{alg:clip}.
The general scheme of the restart algorithm is presented below.

\begin{algorithm}[h]
\caption{  \dm{ \texttt{ZO-Restarts} } }\label{alg:restart}
\begin{algorithmic}[1]
\Procedure{\dm{ \texttt{ZO-Restarts} }}{Algorithm type $\mathcal{A}$, number of restarts $N$, sequence of number of steps $\{T_k\}_{k=1}^N$, sequence of smoothing constants $\{\tau_k\}_{k=1}^N$, sequence of stepsizes $\{\nu_k\}_{k=1}^N$, sequence of clipping constants $\{c_k\}_{k=1}^N$(if necessary), prox-function $\Psi_p$}
    
    \State $x_0 \gets \arg\min\limits_{x \in \mathcal{X}} \Psi_{p}(x)$ or randomly
    
    \For{$k = 0, 1, \dots , N$}  

        \State  Set parameters $\nu_k, (c_k), \Psi_p, \tau_k$ of the Algorithm $\mathcal{A}$
        \State Run $T_k $ iterations of the Algorithm $\mathcal{A}$ with starting point  $x_0$ and get $x_{ \rm final}$
        
        \State $x_0 \gets x_{ \rm final}$

    \EndFor

    \State \textbf{return} $x_{ \rm final}$
\EndProcedure
\end{algorithmic}
\end{algorithm}

\dm{The next theorem provides the convergence guarantee for Algorithm \texttt{ZO-Restarts} run with \texttt{ZO-RSMD}.}


\begin{theorem}\label{restart Conv robust}
Let function $f$ satisfy Assumptions \ref{as: convex}, \ref{as: Lipshcitz and bounded}. Let $\varepsilon >0$ be a fixed accuracy and the $r$-growth Assumption \ref{r-growth} holds with $r \geq \frac{1+\kappa}{\kappa}$. 

Set $R_0 \eqdef \sup_{x,y \in \mathcal{X}} \left(\frac{1+\kappa}{\kappa} D_{\Psi_{p}}(x,y) \right)^\frac{\kappa}{1+\kappa}$ and $R_k = R^0/2^k$.

Set the number of restarts $N = \widetilde{O}\left( \frac{1}{r}\log_2\left(\frac{\mu_rR_0^r}{2\varepsilon}\right)\right)$,   sequence of number of steps $\{T_k\}_{k=1}^N = \left\{  \widetilde{O}\left(\left[ \frac{\sigma_q 2^{(1+r)}}{\mu_r R_k^{r-1}} \right]^\frac{1+\kappa}{\kappa} \right)\right\}_{k=1}^N $, sequence of smoothing constants $\{\tau_k\}_{k=1}^N = \left\{ \frac{\sigma_q R_k }{M_2 T_k^{\frac{\kappa}{1+\kappa}}}\right\}_{k=1}^N$ and sequence of stepsizes $\{\nu_k\}_{k=1}^N = \left\{\frac{R_k^{1/\kappa}}{\sigma_{q} }T_k^{-\frac{1}{1+\kappa}}\right\}_{k=1}^N$, where $\sigma_q$ is given in Lemma \ref{lem: grad norm 1 +k intro}. Finally, let Assumption \ref{as: noize} hold with
$$\Delta_k = \widetilde{O}\left(\frac{\mu_r^2 R_0^{(2r - 1)}}{M_2 \sqrt{d}} \frac{1}{2^{k(2r-1)}}\right), \quad 1 \leq k \leq N.$$
If $x_{ \rm final}$ is the final output  of Algorithm \ref{alg:restart} with \dm{\texttt{ZO-RSMD}} (Algorithm \ref{alg:Robust}) as $\mathcal{A}$ and with the above parameters, then
$$\EE[f(x_{ \rm final})] - f(x^*) \leq \varepsilon,$$
and the total number of steps is
$$T = \widetilde{O}\left( \left[ \frac{a_qM_2 \sqrt{d}}{\mu_r^{1/r}} \cdot \frac{1}{\varepsilon^{\frac{(r-1)}{r}}}\right]^\frac{1+\kappa}{\kappa} \right), \quad a_q \eqdef d^{\frac1q - \frac12}  \min \{ \sqrt{32\ln d - 8} , \sqrt{2q - 1}\},$$
and on the last restart the maximum $\Delta$ threshold is
$$\Delta_N =  \widetilde{O}\left(\frac{\mu_r^{1/r} }{M_2 \sqrt{d}} \varepsilon^{(2-1/r)}\right).$$
 \end{theorem}
\dm{The next theorem provides the convergence guarantee for Algorithm \texttt{ZO-Restarts} run with  \texttt{ZO-Clip-SMD}.}
 \begin{theorem}\label{restart Conv clip}\footnote{In this theorem $\widetilde{O}(\cdot)$ denotes $\log d$ factor for in expectation bounds and $\log d , \log \frac{1}{\delta}$ factors for  in high probability bounds. More explicit formulas are provided in the full proof.} 
Let function $f$ satisfy Assumptions \ref{as: convex}, \ref{as: Lipshcitz and bounded}. Let $\varepsilon >0$ be a fixed accuracy and $r$-growth Assumption \ref{r-growth} holds with $r \geq 2$ for in expectation estimate or $r \geq 1$ for in high probability estimate.
Set $R_0 \eqdef \sup_{x,y \in \mathcal{X}} \left(2 D_{\Psi_{p}}(x,y) \right)^\frac{1}{2}$ and  $R_k = R^0/2^k$.
Set the number of restarts $N = \widetilde{O}\left(\frac{1}{r}\log_2\left(\frac{\mu_rR_0^r}{2\varepsilon}\right)\right)$,   sequence of number of steps $\{T_k\}_{k=1}^N = \left\{\widetilde{O}\left( \left[ \frac{\sigma_q 2^{(1+r)}}{\mu_r R_k^{r-1}} \right]^\frac{1+\kappa}{\kappa} \right)\right\}_{k=1}^N$, sequence of smoothing constants $\{\tau_k\}_{k=1}^N = \left\{ \frac{\sigma_q R_k }{M_2 T_k^{\frac{\kappa}{1+\kappa}}}\right\}_{k=1}^N$, sequence of clipping constants $\{c_k\}_{k=1}^N = \left\{T_k^{\frac{1}{(1+\kappa)}}\sigma_{q}\right\}_{k=1}^N$ and sequence of stepsizes $\{\nu_k\}_{k=1}^N = \left\{\frac{R_k}{c_k}\right\}_{k=1}^N$, where $\sigma_q$ is given in Lemma \ref{lem: grad norm 1 +k intro}. 
Finally, let Assumption \ref{as: noize} hold with
$$\Delta_k = \widetilde{O}\left(\frac{\mu_r^2 R_0^{(2r - 1)}}{M_2 \sqrt{d}} \frac{1}{2^{k(2r-1)}}\right), \quad 1 \leq k \leq N.$$
If $x_{ \rm final}$ is the final output  of Algorithm \ref{alg:restart} with \dm{\texttt{ZO-Clip-SMD}} (Algorithm \ref{alg:clip}) as $\mathcal{A}$ and with the above parameters, then
$$\EE[f(x_{ \rm final})] - f(x^*) \leq \varepsilon,$$
or with probability at least $1-\delta$
$$ f(x_{ \rm final}) - f(x^*) \leq \varepsilon.$$
The total number of steps is
$$T = \begin{cases}\widetilde{O}\left( \left[ \frac{a_qM_2 \sqrt{d}}{\mu_r^{1/r}} \cdot \frac{1}{\varepsilon^{\frac{(r-1)}{r}}}\right]^\frac{1+\kappa}{\kappa} \right),& r > 1\\
\widetilde{O}\left(\left[\frac{a_qM_2 \sqrt{d}}{ \mu_r}\right]^{\frac{1+\kappa}{\kappa}} \log_2\left(\frac{\mu_rR_0}{2\varepsilon}\right)\right),& r = 1\end{cases},$$
$$ a_q \eqdef d^{\frac1q - \frac12}  \min \{ \sqrt{32\ln d - 8} , \sqrt{2q - 1}\},$$
and on the last restart the maximum $\Delta$ threshold is
$$\Delta_N =  \widetilde{O}\left(\frac{\mu_r^{1/r} }{M_2 \sqrt{d}} \varepsilon^{(2-1/r)}\right).$$
 \end{theorem}
 For the complete proofs of Theorems \ref{restart Conv robust}, \ref{restart Conv clip} we refer to \dm{the Appendix} \ref{sec: restarts proofs}.

\subsection{Discussion}
\paragraph{Maximum admissible level of adversarial noise}
\dm{Next, we compare the maximum value of adversarial noise allowed in  \texttt{ZO-RSMD},  \texttt{ZO-Clip-SMD} and \texttt{ZO-Restarts},}
\begin{align*}
     \text{\texttt{ZO-RSMD} (\ref{alg:Robust}) or \texttt{ZO-Clip-SMD} (\ref{alg:clip})}: \qquad &\Delta =  O\left(\frac{\varepsilon^2}{M_2\sqrt{d}\mathcal{D}_{\Psi}} \right),\\
    \text{\texttt{ZO-Restarts} (\ref{alg:restart}) }: \qquad &\Delta =  \widetilde{O}\left(\frac{\mu_r^{1/r}\varepsilon^{(2-1/r)} }{M_2 \sqrt{d}} \right).
\end{align*}
\dm{We notice $r\geq 1$. When $r=1$, the first bound depends on $\varepsilon$ quadratically whereas in the second bound this dependence is linear. } 
When $r$ tends to infinity, the results are the same.
Also, in the beginning $\Delta_k$ can be much bigger and it starts to decrease as $\Delta_k = \frac{\Delta_1}{2^{k(2r-1)}}$ only on \dm{subsequent restarts to reach the required accuracy.}

\paragraph{$q,d, \varepsilon$ dependencies}
\dm{Next, we compare the oracle complexity of \texttt{ZO-RSMD},  \texttt{ZO-Clip-SMD} and \texttt{ZO-Restarts}.}
Again,  \texttt{ZO-Restarts}  guarantees a better  dependence on $\varepsilon$. Below we state the results  in expectation for $r > 1$
\begin{align*}
     \text{\texttt{ZO-RSMD} (\ref{alg:Robust}) or \texttt{ZO-Clip-SMD} (\ref{alg:clip})}: \qquad &T = O\left(\left[ \frac{\sqrt{d} M_2 \mathcal{D}_\Psi a_q}{\varepsilon}\right]^\frac{1+\kappa}{\kappa}\right) ,\\
     \text{\texttt{ZO-Restarts} (\ref{alg:restart}) }: \qquad &T = \widetilde{O}\left( \left[ \frac{\sqrt{d}M_2 a_q }{\mu_r^{1/r} \varepsilon^{\frac{(r-1)}{r}}}  \right]^\frac{1+\kappa}{\kappa} \right).
\end{align*}
In case of $r=1$ \texttt{ZO-Restarts} achieves linear convergence.


\section{Conclusion and Future Work}
In this paper, we proposed and theoretically studied new zeroth-order algorithms to solve non-smooth optimization problems on a convex 
compact set with zeroth-order oracle corrupted by heavy-tailed stochastic noise  (random noise with $(1 + \kappa)$-th bounded moment)
and adversarial noise. We believe that the convergence rates can be improved with the following  possible modifications:  
\begin{enumerate}
    \item 
different sampling strategy for estimating $g_k$, namely  uniform sampling from the unit $\ell_1$-sphere $\{\mathbf{e}: \|\mathbf{e}\|_1 =1 \}$, see, e.g., \cite{akhavan2022gradient}, \cite{lobanov2022gradient}.
\item different assumption about  adversarial noise, namely,  Lipschitz continuity  
$$|\delta(x_1) - \delta(x_2)| \leq M\|x_1 - x_2\|_2, \qquad \forall x_1, x_2 \in \mathcal{X}$$
 see,  e.g.,\cite{dvinskikh2022gradient}. 
\item adaptive strategies and heuristic methods for choosing input parameters of the algorithm, such as stepsize $\nu$, smoothing constant $\tau$, etc. In practice, these constants are difficult to estimate.
\end{enumerate}
We leave their implementation  for future work. We believe that the technique developed in this paper is rather general and makes it possible to use other stochastic gradient methods to obtain new complexity bounds for zeroth-order algorithms.

Also our results can be generalized to obtain the same complexity bounds for non-smooth convex-concave saddle-point problems in terms of the duality gap used in \cite{beznosikov2020gradient}(rather than the gap used in \cite{dvinskikh2022gradient}).\footnote{See the full version of the paper \cite{dvinskikh2022gradient}.}  We leave this for future work.

\section{Acknowledgments}
The work of A. Gasnikov was supported by a grant for research centers in the field of artificial intelligence, provided
by the Analytical Center for the Government of the Russian Federation in accordance with the subsidy agreement
(agreement identifier 000000D730321P5Q0002 ) and the
agreement with the Ivannikov Institute for System Programming of the Russian Academy of Sciences dated November
2, 2021 No. 70-2021-00142.

\bibliography{Ref.bib}

\newpage
\section{Proofs of Lemmas}\label{sec: Lemmas}
\subsection{General results}
\begin{lemma}\label{lem: Jensen for norm}

\begin{enumerate}
    \item 

    For all $x,y \in \RR^{d'}$ and $\kappa \in (0,1]$: \begin{equation}\label{inq}
    \|x-y\|_q^{1+\kappa} \leq 2^\kappa\|x\|_q^{1+\kappa} + 2^\kappa\|y\|_q^{1+ \kappa},
\end{equation}
\item 
\begin{equation}\label{jensen norm less 1}
\forall x,y \geq 0, \kappa \in [0,1]: (x+y)^\kappa \leq x^\kappa + y^\kappa.
\end{equation}
\end{enumerate}
\end{lemma}

\begin{proof}
    \begin{enumerate}
        \item By Jensen's inequality for convex $\|\cdot\|_q^{1+\kappa}$ with ${1+\kappa}> 1$
        $$\|x-y\|_q^{1+\kappa} = 2^{1+\kappa}\|x/2-y/2\|_q^{1+\kappa} \leq2^\kappa\|x\|_q^{1+\kappa} + 2^\kappa\|y\|_q^{1+ \kappa}. $$
        \item Proposition $9$ from \cite{vural2022mirror}.
    \end{enumerate}

\end{proof}
\begin{lemma}\label{lem: Lipschitz f }
    Assumption~\ref{as: Lipshcitz and bounded} implies that $f(x)$ is $M_2$ Lipschitz on $\mathcal{X}$.
\end{lemma}
\begin{proof}
    For all $x,y \in \mathcal{X}$
    \begin{eqnarray}|f(x) - f(y)| &=& |\EE_\xi[f(x,\xi) - f(y,\xi)]| \overset{\text{Jensen's inq}}{\leq}  \EE_\xi[|f(x,\xi) - f(y,\xi)|] \notag \\
    &\leq& \EE_\xi[M_2]\|x-y\|_2 \overset{\text{Jensen's inq}}{\leq} \EE_\xi[M_2^{(1+\kappa)}]^\frac{1}{1+\kappa} \|x-y\|_2 \notag \\
    &\leq&  M_2 \|x-y\|_2. \notag \end{eqnarray}
\end{proof}

\subsection{Smoothing}
\begin{lemma}\label{3L}
Let $f(x)$ be $M_2$ Lipschitz continuous function w.r.t $\|\cdot\|_2$. If $\mathbf{e}$ is random and uniformly distributed on the Euclidean sphere and $\kappa \in (0,1]$, then 
$$\EE_\mathbf{e} \left[ \left(f(\mathbf{e}) - \EE_\mathbf{e} [f(\mathbf{e})] \right)^{2(1 + \kappa)}\right] \leq \left( \frac{bM_2^2}{d}\right)^{1 + \kappa}, \quad b = \frac{1}{\sqrt{2}}.$$
\end{lemma}
\begin{proof}
A standard result of the measure concentration on the Euclidean unit sphere implies that for all $ t > 0$
\begin{equation}\label{measure conc on sphere prob}
    Pr\left(|f(\mathbf{e}) -\EE_\mathbf{e}[f(\mathbf{e})] | > t\right) \leq 2\exp(-b' d t^2/M_2^2), \quad b' = 2 \end{equation}
(see the proof of Proposition 2.10 and Corollary 2.6 in \cite{ledoux2005concentration}). 
Therefore,
\begin{align*}
\EE_\mathbf{e} \left[ \left(f(\mathbf{e}) - \EE_\mathbf{e} [f(\mathbf{e})] \right)^{2(1 + \kappa)}\right] &=\int\limits_{t=0}^\infty Pr\left(|f(\mathbf{e}) -\EE_\mathbf{e}[f(\mathbf{e})] |^{2(1 + \kappa)} > t\right) dt \\ 
&= \int\limits_{t=0}^\infty Pr\left(|f(\mathbf{e}) -\EE_\mathbf{e}[f(\mathbf{e})] | > t^{\frac{1}{2(1 + \kappa)}}\right) dt \\
&\leq \int \limits_{t=0}^\infty 2 \exp\left(-b' d t^{\frac{1}{(1 + \kappa)}}/M_2^2\right) dt  \leq \left(\frac{bM_2^2}{d}\right)^{1+\kappa} .
\end{align*}

\end{proof}
The following lemma gives some useful facts about the measure concentration on the Euclidean unit sphere.
\begin{lemma}\label{upper bounds}
For $q \geq 2, \kappa \in (0,1]$  we get
$$\EE_\mathbf{e} \left[\|\mathbf{e}\|_q^{2(1+\kappa)}\right]  \leq a_{q}^{2(1+\kappa)} \eqdef d^{\frac1q - \frac12}  \min \{ \sqrt{32\ln d - 8} , \sqrt{2q - 1}\}.$$ 
\end{lemma}
This lemma is generalization of Lemma from \cite{gorbunov2019upper} for $\kappa < 1$.
\begin{proof}

We use Lemma $1$  from Theorem $1$ from \cite{gorbunov2019upper} which states that
\begin{enumerate}
    \item Let $e_k$ be $k$-th component of $\mathbf{e}$ then next inequality holds true
\begin{equation}\label{aux}
\EE\left[ |e_k|^q\right] \leq \left( \frac{q-1}{d}\right)^{\frac{q}{2}}, \quad q \geq 2.\end{equation}
\item For any $ x \in \RR^d$  and $q_1 \geq q_2$ we get
\begin{equation}\label{aux2}
\|x\|_{q_1} \leq \|x\|_{q_2},
\end{equation}
\end{enumerate}
We rewrite our objective value as
$$\EE_\mathbf{e}\left[ \|\mathbf{e}\|_q^{2(1+\kappa)} \right] = \EE_\mathbf{e} \left[ \left( \left( \sum\limits_{k=1}^d |e_k|^q \right)^2\right)^\frac{1+\kappa}{q}\right]. $$
Due to Jensen's inequality and equally distributed $e_k$
we obtain
$$ \EE_\mathbf{e} \left[ \left( \left( \sum\limits_{k=1}^d |e_k|^q \right)^2\right)^\frac{1+\kappa}{q}\right]\leq \left( \EE_\mathbf{e}\left[ \left( \sum\limits_{k=1}^d |e_k|^q \right)^2\right]\right)^\frac{1+\kappa}{q}.$$
We use fact that for all $ x_k \geq 0,  k = \overline{1,d} $ 
$$d \sum\limits_{k=1}^d x_k^2 \geq \left( \sum \limits_{k=1}^d x_k\right)^2.$$
Therefore, we estimate
$$\left( \EE_\mathbf{e}\left[ \left( \sum\limits_{k=1}^d |e_k|^q \right)^2\right]\right)^\frac{1+\kappa}{q} \leq \left( d\EE_\mathbf{e}\left[ \sum\limits_{k=1}^d |e_k|^{2q} \right]\right)^\frac{1+\kappa}{q} = (d^2 \EE_\mathbf{e}[|e_k|^{2q}])^\frac{1+\kappa}{q}.$$
Using  \eqref{aux} with $2q$ we continue chain of previous inequalities
$$(d^2 \EE_\mathbf{e}[|e_2|^{2q}])^\frac{1+\kappa}{q} \leq d^{\frac{2(1+\kappa)}{q}} \left( \frac{2q-1}{d}\right)^{1+\kappa} = \left(d^{\frac2q - 1} (2q - 1) \right)^{1+\kappa}.$$ 
Thus, by definition  of $a_{q}$ and obtained estimates we conclude
$$a_{q} = \sqrt{d^{\frac2q - 1} (2q - 1) }.$$
With fixed $d$ and large $q$ more precise upper bound can be obtained. We define function $h_d(q)$ and find its minimum with fixed $d$.
$$h_d(q)  = \ln\left(\sqrt{d^{\frac2q - 1} (2q - 1) }\right) = \left(\frac{1}{q} -\frac12\right)\ln(d) + \frac12 \ln(2q-1),$$
$$\frac{d h_d(q)}{dq} = \frac{-\ln(d)}{q^2} + \frac{1}{2q - 1} = 0,$$
$$q^2 -2\ln(d)q + \ln(d) = 0.$$
When $d \geq 3$ minimal point $q_0$ lies in $[2, +\infty)$
$$q_0 = (\ln d) \left( 1 + \sqrt{1 - \frac{1}{\ln d}}\right), \qquad \ln d \leq q_0 \leq 2\ln d.$$
When $q \geq q_0$ we obtain from \eqref{aux2}
\begin{align*}
    a_{q}  < a_{q_0}  
    &= \sqrt{d^{\frac2q_0 - 1} (2q_0 - 1)} \leq d^{\frac{1}{\ln d} - \frac12} \sqrt{4\ln d - 1} \\
    &=\frac{e}{\sqrt{d}}\sqrt{4\ln d - 1} \leq d^{\frac1q - \frac12}\sqrt{32\ln d - 8},
\end{align*}
Consequently, we get
$$a_{q} = d^{\frac1q - \frac12}  \min \{ \sqrt{32\ln d - 8} , \sqrt{2q - 1}\}. $$

\end{proof}
\begin{lemma}\label{inner product estimate}
For the random  vector $\mathbf{e}$ uniformly distributed on the Euclidean sphere $\{\mathbf{e} \in \RR^d: \|\mathbf{e}\|_2 = 1\}$ and for any $r \in \RR^d$, we have
$$\EE_\mathbf{e}[|\langle \mathbf{e}, r \rangle|] \leq \frac{\|r\|_2}{\sqrt{d}}.$$
\end{lemma}
\begin{lemma}\label{inner product grad r}
Let $g(x, \xi, \mathbf{e})$  be defined in \eqref{g} and $\hat{f}_\tau(x)$ be defined in \eqref{hat_f}. Then, the following holds under Assumption \ref{as: noize}:
$$\EE_{\xi, \mathbf{e}} [\langle g(x,\xi, \mathbf{e}), r \rangle] \geq \langle \nabla \hat{f}_\tau(x) , r \rangle - \frac{d \Delta}{\tau}\EE_\mathbf{e} [|\langle \mathbf{e}, r \rangle|] $$
for any $r \in \RR^d.$
\end{lemma}
\begin{proof}
We remind that by definition \eqref{g} of estimated gradient $g$

$$g(x, \xi, \mathbf{e}) = \frac{d}{2\tau}(f(x + \tau \mathbf{e}, \xi) + \delta(x + \tau \mathbf{e}) - f(x - \tau \mathbf{e}, \xi) - \delta(x - \tau \mathbf{e})) \mathbf{e}.$$
Then  multiplying $g$ on  arbitrary $r$ and  taking full expectation from both sides we get
\begin{eqnarray}\EE_{\xi, \mathbf{e}} [\langle g(x,\xi, \mathbf{e}), r \rangle] &=& \frac{d}{2\tau}\EE_{\xi, \mathbf{e}} [\langle (f(x + \tau \mathbf{e}, \xi) - f(x - \tau \mathbf{e}, \xi))\mathbf{e}, r \rangle]  \notag \\
&+& \frac{d}{2\tau}\EE_{\xi, \mathbf{e}} [\langle (\delta(x + \tau \mathbf{e})  - \delta(x - \tau \mathbf{e}))\mathbf{e}, r \rangle].\notag \end{eqnarray}
In the first term we use fact that $\mathbf{e}$ symmetrically distributed
\begin{eqnarray}
    &&\frac{d}{2\tau}\EE_{\xi, \mathbf{e}} [\langle (f(x + \tau \mathbf{e}, \xi) - f(x - \tau \mathbf{e}, \xi))\mathbf{e}, r \rangle]  
    = \frac{d}{\tau}\EE_{\xi, \mathbf{e}} [\langle f(x + \tau \mathbf{e}, \xi) \mathbf{e}, r \rangle]\notag \\
    &&=\frac{d}{\tau}\EE_{ \mathbf{e}} [\langle \EE_{\xi}[f(x + \tau \mathbf{e}, \xi)] \mathbf{e}, r \rangle] = \frac{d}{\tau} \langle \EE_{ \mathbf{e}} [f(x + \tau \mathbf{e}) \mathbf{e}], r \rangle. \label{eq: lem inner prod proof eq 1}
\end{eqnarray}
Using Lemma \ref{lem: hat_f properties intro} in \eqref{eq: lem inner prod proof eq 1} we take expectation
$$ \frac{d}{\tau} \langle \EE_{ \mathbf{e}} [f(x + \tau \mathbf{e}) \mathbf{e}], r \rangle = \langle \nabla \hat{f}_\tau(x), r \rangle.$$
In the second term we use Assumption \ref{as: noize}
$$\frac{d}{2\tau}\EE_{\xi, \mathbf{e}} [\langle (\delta(x + \tau \mathbf{e})  - \delta(x - \tau \mathbf{e}))\mathbf{e}, r \rangle] \geq - \frac{d\Delta}{\tau} \EE_\mathbf{e}[|\langle \mathbf{e}, r \rangle|].$$
Adding two terms together we get necessary result.

\end{proof}

\begin{proof}[Proof of Lemma \ref{lem: grad norm 1 +k intro}]
By definition \eqref{g} of estimated gradient $g$ we obtain next chain of inequalities
\begin{eqnarray}
&&\EE_{\xi, \mathbf{e}}[\|g(x,\xi,\mathbf{e})\|_q^{1+\kappa}] = \EE_{\xi, \mathbf{e}}\left[\left|\left|\frac{d}{2\tau}(\phi(x + \tau \mathbf{e}, \xi) - \phi(x - \tau \mathbf{e}, \xi)) \mathbf{e}\right|\right|_q^{1+\kappa}\right]  \notag\\
&=& \left(\frac{d}{2\tau} \right)^{1+\kappa}\EE_{\xi, \mathbf{e}}\left[\|\mathbf{e}\|_q^{1+\kappa}|(f(x + \tau \mathbf{e}, \xi) - f(x - \tau \mathbf{e}, \xi) + \delta(x + \tau \mathbf{e}) - \delta(x - \tau \mathbf{e})) |^{1+\kappa}\right]  \notag\\
&\overset{\text{\eqref{inq}}}{\leq}& 2^\kappa\left(\frac{d}{2\tau} \right)^{1+\kappa}\EE_{\xi, \mathbf{e}}\left[\|\mathbf{e}\|_q^{1+\kappa}|f(x + \tau \mathbf{e}, \xi) - f(x - \tau \mathbf{e}, \xi)|^{1+\kappa}\right] \label{eq: grad 1 +k norm first term}\\
&+& 2^\kappa\left(\frac{d}{2\tau} \right)^{1+\kappa} \EE_{\xi, \mathbf{e}}\left[\|\mathbf{e}\|_q^{1+\kappa}|\delta(x + \tau \mathbf{e}) - \delta(x - \tau \mathbf{e}) |^{1+\kappa}\right].\label{eq: grad 1 +k norm second term} 
\end{eqnarray}
Lets deal with \eqref{eq: grad 1 +k norm first term} term. Adding  $\pm\alpha(\xi)$ for any $\alpha(\xi)$ in \eqref{eq: grad 1 +k norm first term} we get 

\begin{align}
&\EE_{\xi, \mathbf{e}}\left[\|\mathbf{e}\|_q^{1+\kappa}|f(x + \tau \mathbf{e}, \xi) - f(x - \tau \mathbf{e}, \xi)|^{1+\kappa}\right] \notag \\
&\leq \EE_{\xi, \mathbf{e}}\left[\|\mathbf{e}\|_q^{1+\kappa}|(f(x + \tau \mathbf{e}, \xi) - \alpha) - (f(x - \tau \mathbf{e}, \xi) - \alpha)|^{1+\kappa}\right]\notag \\
&\overset{\eqref{inq}}{\leq} 2^\kappa\EE_{\xi, \mathbf{e}}\left[\|\mathbf{e}\|_q^{1+\kappa}|f(x + \tau \mathbf{e}, \xi) - \alpha|^{ 1+\kappa} \right] + 2^\kappa\EE_{\xi, \mathbf{e}}\left[\|\mathbf{e}\|_q^{1+\kappa}|f(x - \tau \mathbf{e}, \xi) - \alpha|^{ 1+\kappa} \right]. \label{eq: grad 1 +k technical 1}
\end{align}
We consider that distribution of $\mathbf{e}$ is symmetric, 
\begin{equation}\label{eq: grad 1 +k norm technical 2}
    \text{\eqref{eq: grad 1 +k technical 1}} \leq 2^{\kappa+1}\EE_{\xi, \mathbf{e}}\left[\|\mathbf{e}\|_q^{1+\kappa}|f(x + \tau \mathbf{e}, \xi) - \alpha|^{ 1+\kappa} \right]  . 
\end{equation}
Let $\alpha(\xi)  = \EE_\mathbf{e}[f(x + \tau\mathbf{e}, \xi)]$, then because of Cauchy-Schwartz inequality and conditional expectation properties,
\begin{eqnarray}\text{\eqref{eq: grad 1 +k norm technical 2}}&\leq& 2^{\kappa+1}\EE_{\xi, \mathbf{e}}\left[\|\mathbf{e}\|_q^{1+\kappa}|f(x + \tau \mathbf{e}, \xi) - \alpha|^{ 1+\kappa} \right]  \notag \\
 &=&  2^{\kappa+1}\EE_{\xi} \left[ \EE_\mathbf{e}\left[\|\mathbf{e}\|_q^{1+\kappa}|f(x + \tau \mathbf{e}, \xi) - \alpha|^{ 1+\kappa}\right]\right] \notag \\
&\leq& 2^{\kappa+1}\EE_{\xi}\left[\sqrt{\EE_{\mathbf{e}}\left[\|\mathbf{e}\|_q^{2(1+\kappa)} \right] \EE_\mathbf{e}\left[|f(x + \tau \mathbf{e}, \xi) - \EE_\mathbf{e}[f(x + \tau\mathbf{e}, \xi)]|^{ 2(1+\kappa)} \right]} \right]. \label{eq: grad 1 +k norm techinical 3}
\end{eqnarray}
Next, we use $\EE_\mathbf{e} \left[\|\mathbf{e}\|_q^{2(1+\kappa)} \right] \leq a_{q}^{2(1+\kappa)}$ and Lemma \ref{3L} for $f(x + \tau\mathbf{e}, \xi)$ with fixed $\xi$ and Lipschitz constant $M_2(\xi)\tau$,
\begin{eqnarray} \text{\eqref{eq: grad 1 +k norm techinical 3}} &\leq& 2^{\kappa+1} a_{q}^{1+\kappa} \EE_{\xi}\left[\sqrt{ \left(\frac{2^{-1/2}\tau^2M_2^2(\xi)}{d}\right)^{1 + \kappa}}\right] \notag \\
&=& 2^{\kappa+1} a_{q}^{1+\kappa} \left(\frac{\tau^2 2^{-1/2}}{d}\right)^{(1+\kappa)/2}\EE_{\xi}\left[ M_2^{1+\kappa}(\xi)\right] \notag \\
&\leq& 2^{\kappa+1}\left(\sqrt{\frac{2^{-1/2}}{d}} a_{q}M_2\tau\right)^{1+\kappa}.\label{eq: grad 1+k norm final first term}
\end{eqnarray}

Lets deal with \eqref{eq: grad 1 +k norm second term} term. We use the Cauchy-Schwartz inequality, bounded noise Assumption \ref{as: noize}  and inequality $\EE_\mathbf{e} \left[\|\mathbf{e}\|_q^{2(1+\kappa)}\right]  \leq a_{q}^{2(1+\kappa)}$ that follows from the definition of $a_q$

\begin{align}&\EE_{\xi, \mathbf{e}}\left[\|\mathbf{e}\|_q^{1+\kappa}|\delta(x + \tau \mathbf{e}) - \delta(x - \tau \mathbf{e}) |^{1+\kappa}\right] \notag \\
&\leq \sqrt{\EE_\mathbf{e} \left[\|\mathbf{e}\|_q^{2(1+\kappa)}\right]  \EE_\mathbf{e} \left[|\delta(x + \tau \mathbf{e}) - \delta(x - \tau \mathbf{e}) |^{2(1+\kappa)}\right]} \notag \\
&\leq a_{q}^{1+\kappa} 2^{1+\kappa}\Delta^{1+\kappa} = (2a_{q}\Delta)^{1+\kappa}.\label{eq: grad 1+k norm final second term}
\end{align}
Adding\eqref{eq: grad 1+k norm final first term} and \eqref{eq: grad 1+k norm final second term} we get final result
$$\EE_{\xi, \mathbf{e}}[\|g(x,\xi,\mathbf{e})\|_q^{1+\kappa}] \leq \frac12\left(\frac{d}{\tau} \right)^{1+\kappa} \left(2^{1+\kappa}\left(\sqrt{\frac{2^{-1/2}}{d}} a_{q}\tau M_2\right)^{1+\kappa} + (2a_{q}\Delta)^{1+\kappa} \right) = $$
$$= 2^{\kappa}\left(\frac{\sqrt{d}}{2^{1/4}} a_{q}M_2\right)^{1+\kappa} + 2^{\kappa}\left(\frac{da_{q}\Delta}{\tau}\right)^{1+\kappa}.  $$

\end{proof}

\section{Proof of \texttt{ZO-RSMD} in Expectation Convergence} \label{sec: proof robust conv}

\begin{proof}[Proof of Theorem \ref{Robust Conv}]
By definition $x_* \in \arg\min\limits_{x \in \mathcal{X}} f(x)$.

We use  Convergence Theorem \ref{MD convergence} for Robust SMD Algorithm and set of update vectors $g_k(x_k, \xi_k, \mathbf{e}_k)$
\begin{equation}\label{eq: robust conv SMD conv}
    \frac1T \sum \limits_{k=0}^{T-1} \langle g_{k+1}, x_k  -x^* \rangle \leq \frac{\kappa}{\kappa +1} \frac{R_{0}^{\frac{1+\kappa}{\kappa}} }{\nu T} + \frac{\nu^{\kappa}}{1+\kappa} \frac{1}{T} \sum \limits_{k=0}^{T-1} \|g_{k+1}\|^{1+\kappa}_q.
\end{equation}
Then we take  full expectation $\EE$ from both sides of \eqref{eq: robust conv SMD conv}
\begin{equation} \label{eq: conv robust start inq}  
\frac1T \sum \limits_{k=0}^{T-1} \EE\left[  \langle g_{k+1}, x_k  -x^* \rangle\right] \leq \frac{\kappa}{\kappa +1} \frac{R_{0}^{\frac{1+\kappa}{\kappa}} }{\nu T} + \frac{\nu^{\kappa}}{1+\kappa}\frac{1}{T} \sum \limits_{k=0}^{T-1}\EE\left[\|g_{k+1}\|^{1+\kappa}_q \right].\end{equation}
Using boundness of estimated gradient $(1+\kappa)$-th moment from Lemma \ref{lem: grad norm 1 +k intro} for the right part of inequality \eqref{eq: conv robust start inq} we get
\begin{equation}\label{eq: conv robust part 1}
    \frac{\nu^{\kappa}}{1+\kappa}\frac{1}{T} \sum \limits_{k=0}^{T-1}\EE\left[\|g_{k+1}\|^{1+\kappa}_q \right] \leq \frac{\nu^{\kappa}}{1+\kappa}\frac{1}{T} \sum \limits_{k=0}^{T-1} \sigma_{q}^{1+\kappa} \leq \frac{\nu^{\kappa}}{1+\kappa}\sigma_{q}^{1+\kappa}. \end{equation}
Using conditional math expectation and Lemma \ref{inner product grad r} for the left part of inequality \eqref{eq: conv robust start inq} we estimate

\begin{align}\label{200}
&\frac1T \sum \limits_{k=0}^{T-1} \EE\left[  \langle g_{k+1}, x_k  -x^* \rangle\right] = \frac1T \sum \limits_{k=0}^{T-1} \EE\left[  \EE_{|\leq k} [\langle g_{k+1}, x_k  -x^* \rangle]\right]\notag \\
&\geq \frac1T \sum \limits_{k=0}^{T-1}\EE[\langle \nabla \hat{f}_\tau(x_{k}) , x_k  -x^* \rangle] - \frac1T \sum \limits_{k=0}^{T-1}\frac{d \Delta}{\tau}\EE\left[\EE_{\mathbf{e}_k|\leq k} [|\langle \mathbf{e}_k, x_k  -x^* \rangle|]\right].
\end{align}

\begin{enumerate}
    \item 

For the first term of \eqref{200} by convexity of $\hat{f}_\tau(x)$ we obtain
$$\frac1T \sum \limits_{k=0}^{T-1}\EE[\langle \nabla \hat{f}_\tau(x_{k}) , x_k  -x^* \rangle] \geq \frac1T \sum \limits_{k=0}^{T-1} \left( \EE[\hat{f}_\tau(x_k)] - \hat{f}_\tau(x_*)\right).$$
Then we define $\overline{x}_T  = \frac1T \sum \limits_{k=0}^{T-1} x_k$ and use Jensen's inequality
$$\frac1T \sum \limits_{k=0}^{T-1} \left( \EE[\hat{f}_\tau(x_k)] - \hat{f}_\tau(x_*)\right) \geq \EE[\hat{f}_\tau(\overline{x}_T)] - \hat{f}_\tau(x^*).$$
Finally, we apply  approximation property of $\hat{f}_\tau(x)$ from Lemma  \ref{lem: hat_f properties intro}
\begin{equation}\label{eq: conv robust part 2}
    \EE[\hat{f}_\tau(\overline{x}_T)] - \hat{f}_\tau(x^*) \geq \EE[f(\overline{x}_T)] - f(x^*) - 2M_2\tau.\end{equation}
\item For the second term of \eqref{200} we use concentration measure property  from Lemma~\ref{inner product estimate} and estimate
 \begin{eqnarray}
 - \frac{d \Delta}{T\tau} \sum \limits_{k=0}^{T-1}\EE_{\mathbf{e}_k|\leq k} [|\langle \mathbf{e}_k, x_k  -x^* \rangle|] 
\geq -\frac{d\Delta}{T\tau}\sum \limits_{k=0}^{T-1} \frac{1}{\sqrt{d}}\|x_k - x^*\|_2 \notag \\
\overset{p \leq 2}{\geq} -\frac{d\Delta}{T\tau}\sum \limits_{k=0}^{T-1} \frac{1}{\sqrt{d}} \|x_k - x^*\|_p .\label{eq: conv robust tech 1}
\end{eqnarray}
Let's notice that $\Psi_{p}$ is $\left( 1, \frac{1+\kappa}{\kappa}\right)$-uniformly convex function w.r.t. $p$ norm. Then by definition \eqref{def: uniform convex} we bound $\|x_k - x^*\|_{p}$
$$\|x_k - x^*\|_{p} \leq \left(\frac{1+\kappa}{\kappa} D_{\Psi_{p}}(x_k, x^*) \right)^\frac{\kappa}{1+\kappa} \leq \sup_{x,y \in \mathcal{X}} \left(\frac{1+\kappa}{\kappa} D_{\Psi_{q^*}}(x, y) \right)^\frac{\kappa}{1+\kappa} = \mathcal{D}_\Psi $$
Hence, after this bound we get
\begin{equation}\label{eq: conv robust part 3}
    \text{\eqref{eq: conv robust tech 1}} \geq -\frac{d\Delta}{T\tau}\sum \limits_{k=0}^{T-1} \frac{1}{\sqrt{d}} \|x_k - x^*\|_{p} \geq -\frac{\sqrt{d}\Delta}{\tau}\mathcal{D}_\Psi . \end{equation}
\end{enumerate}

Next, we combine \eqref{eq: conv robust part 1}, \eqref{eq: conv robust part 2}, \eqref{eq: conv robust part 3} together to obtain final estimate 

\begin{equation}\label{smd before nu}
\EE[f(\overline{x}_T)] - f(x^*) \leq 2M_2\tau + \frac{\sqrt{d}\Delta}{\tau}\mathcal{D}_\Psi+ \frac{R_{0}^{\frac{1+\kappa}{\kappa}} }{\nu T} + \frac{\nu^{\kappa}}{1+\kappa}\sigma_{q}^{1+\kappa}.
\end{equation}
Now we select good parameters of the Algorithm to lower right part of \eqref{smd before nu}.
By choosing optimal $\nu = \frac{R_{0}^{1/\kappa}}{\sigma_{q} }T^{-\frac{1}{1+\kappa}}$ we get
$$\EE[f(\overline{x}_T)] - f(x^*) \leq 2M_2\tau + \frac{\sqrt{d}\Delta}{\tau}\mathcal{D}_\Psi  + 2R_{0}\sigma_{q} T^{-\frac{\kappa}{1+\kappa}}.$$
Finally, we get explicit bound of $\sigma_q$ using Lemma \ref{lem: Jensen for norm}
$$\sigma_q \leq 2\left(\frac{\sqrt{d}}{2^{1/4}} a_{q}M_2\right) + 2\left(\frac{da_{q}\Delta}{\tau}\right).$$
And set optimal $\tau$ 
$$\tau =  \sqrt{\frac{\sqrt{d}\Delta\mathcal{D}_{\Psi} + 4R_{0}da_{q}\Delta T^{-\frac{\kappa}{1+\kappa}}}{2M_2}}.$$

\end{proof}

\section{Proof of \texttt{ZO-Clip-SMD} in Expectation Convergence}\label{sec: clip conv}
First, we prove some useful statements about clipped gradient vector properties. Similar proof can be found in \cite{zhang2022parameter}.

    
    
\begin{proof}[Proof of Lemma \ref{clip_grad_properties intro}]

\begin{enumerate}
\item 

By Jensen's inequality for $\|\cdot\|_q$ and definition of $\hat{g}$ we estimate
\begin{eqnarray}
\|\hat{g} - \EE[\hat{g} ]\|_q  &\leq& \|\hat{g}\|_q +  \|\EE[\hat{g} ]\|_q \notag \\
&\leq& \left|\left|\frac{g}{\|g\|_q} \min(\|g\|_q, c)\right|\right|_q +  \EE\left[\left|\left|\frac{g}{\|g\|_q} \min(\|g\|_q, c) \right|\right|_q\right] \notag \\
&=& \min(\|g\|_q, c) + \EE[\min(\|g\|_q, c)] \notag \\
&\leq& c + c  = 2c. \notag
\end{eqnarray}

\item
\begin{enumerate}
    \item 
    Considering $\EE[\|g(x,\xi,\mathbf{e})\|_q^{1+\kappa}]\leq  \sigma_{q}^{1+\kappa}$ and $\|\hat{g}\|_q \leq c$ we get
    $$\EE[\|\hat{g}\|^{1+\kappa}_q \|\hat{g}\|^{1-\kappa}_q]\leq \sigma_{q}^{1+\kappa} c^{1-\kappa}.$$

    \item By Jensen's inequality for $\|\cdot\|_q$ we obtain
    \begin{eqnarray}
   \EE[\|\hat{g} - \EE[\hat{g}]\|^2_q] &\leq& 2\EE[\|\hat{g}\|_q^2 + 2\|\EE[\hat{g}]\|^2_q ] \notag \\
  & \leq& 2\EE[\|\hat{g}\|_q^2] + 2\EE[\|\hat{g}\|^2_q] ] \notag \\
    &\overset{\eqref{clip grad  second moment}}{\leq}& 2\sigma_{q,\kappa}^{1+\kappa} c^{1-\kappa} + 2\sigma_{q,\kappa}^{1+\kappa} c^{1-\kappa} \leq 4\sigma_{q,\kappa}^{1+\kappa} c^{1-\kappa}.\notag
    \end{eqnarray}
\item

    Due to convexity of norm function and Jensen's inequality we estimate
    $$\|\EE[g] - \EE [\hat{g}]\|_q \leq \EE[\|g - \hat{g}\|_q] \leq \EE\left[\|g\|_q \mathds{1}_{\{\|g\|_q > c \}}\right].$$
    Final result follows from $\|g\|^{1+\kappa}_q \mathds{1}_{\{\|g\|_q > c \}} \geq \|g\|_q c^{\kappa} \mathds{1}_{\{\|g\|_q > c \}}$ 
    \begin{equation}\label{bayes add}\EE\left[\|g\|_q \mathds{1}_{\{\|g\|_q > c \}}\right] \leq\EE\left[\|g\|_q \mathds{1}_{\{\|g\|_q > c \}}\right] \leq \frac{\sigma_{q,\kappa}^{1+\kappa}}{c^\kappa}.\notag \end{equation}

\end{enumerate}
\end{enumerate}

\end{proof}

\begin{proof}[Proof of the Theorem \ref{Clip Conv}]
Let us note from first term of \eqref{200} in the proof of  Theorem \ref{Robust Conv}   that for any $x_k$
\begin{equation}\label{loss}
     f(\overline{x}_T) - f(x^*)  \leq \frac1T \sum \limits_{k=0}^{T-1}\langle \nabla \hat{f}_\tau(x_{k}) , x_k  -x^* \rangle +  2M_2\tau. \end{equation}
Next, we  define  functions 
$$l_k(x) \eqdef \langle \EE_{|\leq k}[\hat{g}_{k+1}], x  - x^* \rangle. $$
Note that $l_k(x)$ is convex for any $k$ and $\nabla l_k(x) = \EE_{|\leq k}[\hat{g}_{k+1}]$. Therefore sampled estimation gradient is unbiased.
With these functions we can rewrite the right part of \eqref{loss} as follows
\begin{align}\label{D and E terms in clip no exp}
&\frac1T \sum \limits_{k=0}^{T-1}\langle \nabla \hat{f}_\tau(x_{k}) , x_k  -x^* \rangle +  2M_2\tau \notag \\
    &=  \frac1T \sum \limits_{k=0}^{T-1}\left(\langle \nabla \hat{f}_\tau(x_{k}) - \EE_{|\leq k}[\hat{g}_{k+1}] , x_k  -x^* \rangle   \right) +  \frac1T \sum \limits_{k=0}^{T-1}\left( l_k(x_k) - l_k(x^*)\right) + 2M_2\tau.\end{align}
Then we take full expectation from both sides of \eqref{D and E terms in clip no exp}

\begin{eqnarray}\label{D and E terms in clip}
&\EE & \left[\frac1T \sum \limits_{k=0}^{T-1}\langle \nabla \hat{f}_\tau(x_{k}) , x_k  -x^* \rangle \right]+  2M_2\tau \notag \\
    &=& \underbrace{ \EE \left[\frac1T \sum \limits_{k=0}^{T-1}\left(\langle \nabla \hat{f}_\tau(x_{k}) - \EE_{|\leq k}[\hat{g}_{k+1}] , x_k  -x^* \rangle   \right) \right]}_\text{D} \notag \\
    &+&  \underbrace{\EE\left[\frac1T \sum \limits_{k=0}^{T-1}\left( l_k(x_k) - l_k(x^*)\right)\right]}_\text{E} + 2M_2\tau. \notag \end{eqnarray}
We add $\pm \EE_{|\leq k}[g_{k+1}]$ in D  term and get
\begin{align}\label{eq: conv clip D term sum}
&\EE\left[\frac1T \sum \limits_{k=0}^{T-1}\left(\langle \nabla \hat{f}_\tau(x_{k}) - \EE_{|\leq k}[\hat{g}_{k+1}] , x_k  -x^* \rangle   \right)\right]\notag \\
    &= \EE\left[\frac1T \sum \limits_{k=0}^{T-1} \langle  \EE_{|\leq k}[g_{k+1}] - \EE_{|\leq k}[\hat{g}_{k+1}] , x_k  -x^* \rangle  \right] \notag\\
    &+ \EE\left[\frac1T \sum \limits_{k=0}^{T-1}\langle \nabla \hat{f}_\tau(x_{k}) - \EE_{|\leq k}[g_{k+1}] , x_k  -x^* \rangle \right].  \end{align}
In order to bound the first term of \eqref{eq: conv clip D term sum} let's notice that $\Psi_{p}$ is $\left( 1, 2\right)$-uniformly convex function w.r.t. $p$ norm. Then by definition \eqref{def: uniform convex} we bound $\|x_k - x^*\|_{p}$

$$\|x_k - x^*\|_{p} \leq \left(2 D_{\Psi_{p}}(x_k, x^*) \right)^\frac12 \leq \sup_{x,y \in \mathcal{X}} \left(2 D_{\Psi_{p}}(x, y) \right)^\frac12 = \mathcal{D}_\Psi, $$
and estimate $\|x_k - u\|_{p} \leq \mathcal{D}_\Psi,$  for all $  u \in\mathcal{X}.$

We apply the Cauchy–Schwarz inequality to inner product in the first term of \eqref{eq: conv clip D term sum}
\begin{align} \label{eq: conv clip part 1}
  &\EE\left[\frac1T \sum \limits_{k=0}^{T-1}\left(\langle  \EE_{|\leq k}[g_{k+1}] - \EE_{|\leq k}[\hat{g}_{k+1}] , x_k  -x^* \rangle   \right)\right]\notag \\
  &\leq  \frac1T \sum_{k=0}^{T-1} \left( \EE\left[\EE_{|\leq k}\left[\|\EE_{|\leq k}[g_{k+1}] - \EE_{|\leq k}[\hat{g}_{k+1}]\|_q\|x_k- x^*\|_{p}\right] \right] \right)   \overset{\eqref{clip grad nabla f tau diff}}{\leq} \mathcal{D}_\Psi\frac{\sigma_{q}^{1+\kappa}}{c^\kappa}.\end{align}
To bound the second term in \eqref{eq: conv clip D term sum} we use Lemma \ref{inner product grad r} and Lemma \ref{inner product estimate}
\begin{align}\label{eq: conv clip part 2}
    &\EE\left[\frac1T \sum \limits_{k=0}^{T-1}\left(\langle \nabla \hat{f}_\tau(x_{k}) - \EE_{|\leq k}[g_{k+1}] , x_k  -x^* \rangle   \right)\right] \notag \\
    &\leq \frac1T \sum \limits_{k=0}^{T-1} \frac{d \Delta}{\tau}\EE\left[\EE_{\mathbf{e}|<k}[|\langle \mathbf{e},x_k - x^*\rangle|]\right] \notag \\
    &\leq \frac1T \sum \limits_{k=0}^{T-1} \frac{d \Delta}{\tau} \frac{1}{\sqrt{d}}\EE[\|x_k - x^*\|_2] \notag \\
    & \overset{p\leq 2}{\leq} \frac1T \sum \limits_{k=0}^{T-1} \frac{d \Delta}{\tau} \frac{1}{\sqrt{d}}\EE[\|x_k - x^*\|_p] \leq \frac{\Delta \sqrt{d}}{\tau} \mathcal{D}_\Psi.
\end{align}

Next, we bound E term. First of all, we rewrite it as

\begin{align*}
    \EE \left[\frac1T \sum \limits_{k=0}^{T-1}\left( l_k(x_k) - l_k(x^*)\right) \right] &= \frac1T \sum \limits_{k=0}^{T-1} \EE \left[\EE_{|\leq k}[\langle \EE_{|\leq k}[\hat{g}_{k+1}], x_k  -x^* \rangle] \right] \\
    &= \frac1T \sum \limits_{k=0}^{T-1} \EE \left[\EE_{|\leq k}[\langle \hat{g}_{k+1}, x_k  -x^* \rangle]\right].
\end{align*}

For the Robust SMD Algorithm with update vectors $\hat{g_k}  $ by Convergence Theorem \ref{MD convergence} with bounded second moment next inequality holds true
\begin{equation}\label{conv from E clip proof}
    \frac1T \sum \limits_{k=0}^{T-1} \langle \hat{g}_{k+1}, x_k  -x^* \rangle \leq \frac12 \frac{R_{0}^{2} }{\nu T} + \frac{\nu}{2} \frac{1}{T} \sum \limits_{k=0}^{T-1} \|\hat{g}_{k+1}\|^{2}_q.
    \end{equation}
Taking $\EE $ from both sides of \eqref{conv from E clip proof} we get
\begin{eqnarray} 
\frac1T \sum \limits_{k=0}^{T-1} \EE [ \langle \hat{g}_{k+1}, x_k  -x^* \rangle] &=& \frac1T \sum \limits_{k=0}^{T-1} \EE \left[ \EE_{|\leq k}[\langle \hat{g}_{k+1}, x_k  -x^* \rangle]\right] \notag \\
&\leq& \frac12 \frac{R_{0}^{2} }{\nu T} + \frac{\nu}{2} \frac{1}{T} \sum \limits_{k=0}^{T-1} \EE\left[\EE_{|\leq k}[\|\hat{g}_{k+1}\|^{2}_q]\right]. \notag
\end{eqnarray}
By \eqref{clip grad  second moment} from Lemma \ref{clip_grad_properties intro} we bound second moment of clipped gradient 

$$\EE_{|\leq k}(\|\hat{g}_{k+1}\|^2_q)  \leq \sigma_{q}^{1+\kappa} c^{1 - \kappa}, $$
And hence get,
\begin{equation}\label{eq: conv clip part 3}
    \frac1T \sum \limits_{k=0}^{T-1} \EE [ \langle \hat{g}_{k+1}, x_k  -x^* \rangle] \leq \frac12 \frac{R_{0}^{2} }{\nu T} + \frac{\nu}{2}  \sigma_{q}^{1+\kappa} c^{1 - \kappa}. \end{equation} 
Combining bounds \eqref{eq: conv clip part 1}, \eqref{eq: conv clip part 2}, \eqref{eq: conv clip part 3} together, we obtain final estimate
$$\EE [f(\overline{x}_T)] - f(x^*)  \leq 2M_2\tau + \frac12 \frac{R_{0}^{2} }{\nu T} + \frac{\nu}{2}  \sigma_{q}^{1+\kappa} c^{1 - \kappa} + \left(\frac{\sigma_{q}^{1+\kappa}}{c^\kappa} + \Delta \frac{\sqrt{d}}{\tau}\right) \mathcal{D}_\Psi. $$
In order to get minimal upper bound we find optimal parameters. First, we choose $c$ by finding minimum of 
$$  \min_{c>0}\sigma_{q}^{1+\kappa} \left( \frac{1}{c^\kappa}\mathcal{D}_\Psi + \frac{\nu}{2}   c^{1 - \kappa} \right) = \min_c \sigma_{q}^{1+\kappa} h_1(c)$$

$h_1'(c) = \frac{\nu}{2} (1 - \kappa) c^{-\kappa} - \kappa\frac{1}{c^{1+\kappa}}\mathcal{D}_\Psi = 0 \Rightarrow c^* =  \frac{2\kappa\mathcal{D}_\Psi}{(1-\kappa)\nu}.$

\begin{eqnarray}
    \EE [f(\overline{x}_T)] - f(x^*)  &\leq& 2M_2\tau + \frac12 \frac{R_{0}^{2} }{\nu T} + \Delta \frac{\sqrt{d}}{\tau} \mathcal{D}_\Psi \notag \\
    &+& \sigma_{q,\kappa}^{1+\kappa} \left( \mathcal{D}^{1-\kappa} 2^{-\kappa} \nu^{\kappa} \left[ \frac{(1-\kappa)^\kappa}{\kappa^\kappa} + \frac{\kappa^{(1-\kappa)}}{(1-\kappa)^{(1-\kappa)}}\right]\right). \label{eq: for simplify}
\end{eqnarray}
Considering bound of $\kappa \in [0,1]$ and as consequence
$$\left[ \frac{(1-\kappa)^\kappa}{\kappa^\kappa} + \frac{\kappa^{(1-\kappa)}}{(1-\kappa)^{(1-\kappa)}}\right] \leq 2, $$
we simplify \eqref{eq: for simplify}
\begin{equation}\label{clip before nu}
    \EE [f(\overline{x}_T)] - f(x^*)  \leq 2M_2\tau + \frac12 \frac{R_{0}^{2} }{\nu T} + \Delta \frac{\sqrt{d}}{\tau} \mathcal{D}_\Psi + \sigma_{q}^{1+\kappa} \left( 2\mathcal{D}_\Psi^{1-\kappa}  \nu^{\kappa} \right).\end{equation}
Choosing optimal $\nu^*$ similarly we get
$$\nu^* = \left( \frac{R_{0}^2}{4T\kappa\sigma_{q}^{1+\kappa} \mathcal{D}_\Psi^{1-\kappa} }\right)^{\frac{1}{1+\kappa}}$$
And
$$\EE [f(\overline{x}_T)] - f(x^*)  \leq 2M_2\tau + \Delta \frac{\sqrt{d}}{\tau} \mathcal{D}_\Psi+ \frac{R_{0}^\frac{2\kappa}{1+\kappa} \mathcal{D}_\Psi^{\frac{1-\kappa}{1+\kappa}}\sigma_{q} }{T^\frac{\kappa}{1+\kappa}} 2\left[\kappa^\frac{1}{1+\kappa} + \kappa^{-\frac{\kappa}{1+\kappa}} \right].$$
Considering bound of $\kappa \in [0,1]$ next inequality holds true
$$\left[\kappa^\frac{1}{1+\kappa} + \kappa^{-\frac{\kappa}{1+\kappa}} \right] \leq 2.$$
Thus, we can  simplify upper bound even more
\begin{equation}\label{after nu}
    \EE [f(\overline{x}_T)] - f(x^*)  \leq 2M_2\tau + \Delta \frac{\sqrt{d}}{\tau} \mathcal{D}_\Psi+2 \frac{R_{0}^\frac{2\kappa}{1+\kappa} \mathcal{D}_\Psi^{\frac{1-\kappa}{1+\kappa}}\sigma_{q} }{T^\frac{\kappa}{1+\kappa}}. \end{equation}
In order to avoid $\nu \rightarrow \infty$ when $\kappa \rightarrow 0$ one can also choose $\nu^* = \left( \frac{R_{0}^2}{4T\sigma_{q}^{1+\kappa} \mathcal{D}_\Psi^{1-\kappa} }\right)^{\frac{1}{1+\kappa}}$. Estimation \eqref{after nu} does not change.

Finally, we get explicit bound $\sigma_q$ with Lemma \ref{lem: Jensen for norm}
$$\sigma_q \leq 2\left(\frac{\sqrt{d}}{2^{1/4}} a_{q}M_2\right) + 2\left(\frac{da_{q}\Delta}{\tau}\right),$$
And set optimal $\tau$
$$\tau = \sqrt{\frac{\sqrt{d}\Delta\mathcal{D}_{\Psi} + 4R_{0}^\frac{2\kappa}{1+\kappa} \mathcal{D}_\Psi^{\frac{1-\kappa}{1+\kappa}}da_{q}\Delta T^{-\frac{\kappa}{1+\kappa}}}{2M_2}}.$$

\end{proof}
\section{Proof of \texttt{ZO-Clip-SMD} in High Probability Convergence} \label{sec: clip conv prob}

Before we turn to the proof of Theorem \ref{Clip Conv prob}, we present two lemmas representing the Bernstein inequality for the sum of martingale differences and the sum of squares of bounded random variables.
These are some classical results on  measure concentration. 
\begin{lemma}[Lemma $23$ from \cite{zhang2022parameter}]\label{Bernstein inequality}
    Let $\{X_i\}_{i \geq 1}$ be martingale difference sequence, i.e. $\EE[X_i| X_{i-1}, \dots, X_1] = 0$ for all $i \geq 1$. Also $b, \sigma$ is such deterministic constants that $|X_i| < b$ almost surely and $  \EE[X_i^2| X_{i-1}, \dots, X_1] < \sigma^2$  for $ i \geq 1$. Then for arbitrary fixed number $\mu$ and for all $T$ with probability at least $1 - \delta$ next inequality holds true
    $$\left|\sum\limits_{i=1}^t \mu X_i \right| \leq 2b|\mu|\log\frac1\delta + \sigma|\mu|\sqrt{2T\log\frac1\delta}.$$
\end{lemma}
\begin{lemma}[Theorem $20$ from \cite{zhang2022parameter}] \label{Sum of squares MSD}
Let $Z_i$ be a sequence of random variables adapted to a filtration $\mathcal{F}_t$. Further, suppose $|Z_i| < b$ almost surely and $\EE[Z_i^2] \leq \sigma^2$ . Then for any $\mu > 0$ with probability at least $1 -\delta$ next inequality holds true
\begin{eqnarray}
    \sum\limits_{k=1}^T Z_k^2 &\leq& 3T\sigma^2 \log\left( \frac4\delta\left[ \log\left(\sqrt{\frac{\sigma^2T}{\mu^2}}\right)+2\right]^2\right)\notag \\
    &+&20\max(\mu^2, b^2)\log\left(\frac{112}{\delta}\left[ \log\left(\frac{2\max(\mu,b)}{\mu}\right) + 1\right]^2\right). \notag
    \end{eqnarray}
By choosing $\mu = b \geq \sigma$ we simplify
$$\sum\limits_{k=1}^T Z_k^2 \leq 3T\sigma^2 \log\left( \frac4\delta\left[ \log\left(\sqrt{T}\right)+2\right]^2\right)+20b^2\log\left(\frac{12}{\delta}\right).$$

\end{lemma}

\begin{proof}[Proof of the Theorem \ref{Clip Conv prob}]
Lets notice from the first term of \eqref{200} in the proof of  Theorem \ref{Robust Conv}  that for any $x_k$ next inequality holds true

\begin{equation}\label{loss high prob}
     f(\overline{x}_T) - f(x^*)  \leq \frac1T \sum \limits_{k=0}^{T-1}\langle \nabla \hat{f}_\tau(x_{k}) , x_k  -x^* \rangle +  2M_2\tau. \end{equation}
For the Robust SMD Algorithm with update vectors $\hat{g_k}  $  Convergence Theorem \ref{MD convergence} with the bounded second moment guarantees that
\begin{equation}\label{eq: conv clip prob main eq}
    \frac1T \sum \limits_{k=0}^{T-1} \langle \hat{g}_{k+1}, x_k  -x^* \rangle \leq \frac12 \frac{R_{0}^{2} }{\nu T} + \frac{\nu}{2} \frac{1}{T} \sum \limits_{k=0}^{T-1} \|\hat{g}_{k+1}\|^{2}_q. 
\end{equation} 
Let's define random variable $Z_k = \|\hat{g}_{k+1}\|_q$ and notice that $|Z_k| \leq c$ by definition of clipping and $\EE[Z_i^2] \leq 4\sigma_{q,\kappa}^{1+\kappa} c^{1 - \kappa}$ by  \eqref{clip grad  var} from clipped gradient properties Lemma~\ref{clip_grad_properties intro}. Thus, we can apply  Lemma~\ref{Sum of squares MSD} and with probability at least $1 - \delta$ bound mean sum of the clipped gradients second moments  
\begin{equation}\label{eq: conv clip prob part 1}
    \frac{1}{T} \sum \limits_{k=0}^{T-1} \|\hat{g}_{k+1}\|^{2}_q \leq 12\sigma_{q,\kappa}^{1+\kappa} c^{1 - \kappa} \log\left( \frac4\delta\left[ \log\left(\sqrt{T}\right)+2\right]^2\right) + \frac{20}{T}c^2\log\left(\frac{12}{\delta}\right).\end{equation}
The left part of \eqref{eq: conv clip prob main eq} can be rewritten as
\begin{align*}
  &\frac1T \sum \limits_{k=0}^{T-1}\langle \hat{g}_{k+1}, x_k  -x^* \rangle = \frac1T \sum \limits_{k=0}^{T-1}\langle \hat{g}_{k+1} - \nabla \hat{f}_\tau(x_k), x_k  -x^* \rangle + \frac1T \sum \limits_{k=0}^{T-1}\langle \nabla \hat{f}_\tau(x_k), x_k  -x^* \rangle \\
  &=\underbrace{\frac1T \sum \limits_{k=0}^{T-1}\langle \hat{g}_{k+1} - \EE_{|\leq k}[\hat{g}_{k+1}]  , x_k  -x^* \rangle}_\text{\circled{1}} + \underbrace{\frac1T \sum \limits_{k=0}^{T-1}\langle \EE_{|\leq k}[\hat{g}_{k+1}] - \nabla \hat{f}_\tau(x_k), x_k  -x^* \rangle}_\text{\circled{2}}\\
  &+ \underbrace{\frac1T \sum \limits_{k=0}^{T-1}\langle \nabla \hat{f}_\tau(x_k), x_k  -x^* \rangle}_\text{\circled{3}}. 
\end{align*}

In the \circled{1} term we can proof that this is the sum of the martingale sequence difference. Indeed, we notice that $x_k$ is fixed when we take $\EE_{|\leq k}$ and martingale property holds true, i.e.
$$\EE_{|\leq k}[\langle \hat{g}_{k+1} - \EE_{|\leq k}[\hat{g}_{k+1}], x_k  - x^* \rangle] = 0.$$
By  Lemma \ref{clip_grad_properties intro} we bound each element of martingale sequence 
$$|\langle \hat{g}_{k+1} - \EE_{|\leq k}[\hat{g}_{k+1}], x_k  -x^* \rangle| \leq \|\hat{g}_{k+1} - \EE_{|\leq k}[\hat{g}_{k+1}]\|_q \|x_k  -x^*\|_{p} \leq 2c \cdot \|x_k  -x^*\|_{p}.$$
And by  \eqref{clip grad  var} from Lemma \ref{clip_grad_properties intro}  we bound expectation of square of each element
$$\EE\left[|\langle \hat{g}_{k+1} - \EE_{|\leq k}[\hat{g}_{k+1}], x_k  -x^* \rangle|^2\right] \leq 4\sigma_{q}^{1+\kappa} c^{1 - \kappa} \cdot \|x_k  -x^*\|^2_{p}.$$
Lets notice that $\Psi_{p}$ is $\left( 1, 2\right)$-uniformly convex function w.r.t. $p$ norm. Then by definition \eqref{def: uniform convex} we bound

$$\|x_k - x^*\|_{p} \leq \left(2 D_{\Psi_{p}}(x_k, x^*) \right)^\frac12 \leq \sup_{x,y \in \mathcal{X}} \left(2 D_{\Psi_{p}}(x, y) \right)^\frac12 = \mathcal{D}_\Psi, $$
and  estimate $\|x_k - u\|_{p} \leq \mathcal{D}$  for all $ u \in\mathcal{X}$.
Hence, we can apply  Bernstein's inequality Lemma \ref{Bernstein inequality} and get  with probability at least $1 - \delta$ and $\mu = \frac1T$ that
\begin{equation}\label{eq: conv clip prob part 2}\frac1T \sum \limits_{k=0}^{T-1}|\langle \hat{g}_{k+1} - \EE_{|\leq k}[\hat{g}_{k+1}]  , x_k  -x^* \rangle| \leq \frac{4c \mathcal{D}_\Psi}{T}\log\frac1\delta + \frac{\sqrt{4\sigma_{q}^{1+\kappa} c^{1 - \kappa}}}{\sqrt{T}} \mathcal{D}_\Psi^2 \sqrt{2\log \frac1\delta}.\end{equation}
For the \circled{2} we use bound of D term from \eqref{D and E terms in clip} in the proof of Theorem~\ref{Clip Conv}
\begin{equation}\label{eq: conv clip prob part 3}|\langle \EE_{|\leq k}[\hat{g}_{k+1}] - \nabla \hat{f}_\tau(x_k), x_k  -x^* \rangle| \leq\left( \frac{\sigma_{q}^{1+\kappa}}{c^\kappa} + \Delta\frac{\sqrt{d}}{\tau}\right)\mathcal{D}_\Psi.\end{equation}
For the \circled{3} we use already obtained bound \eqref{loss high prob}
\begin{equation}\label{eq: conv clip prob part 4}f(\overline{x}_T) - f(x^*) - 2M_2\tau \leq \frac1T \sum \limits_{k=0}^{T-1}\langle \nabla \hat{f}_\tau(x_{k}) , x_k  -x^* \rangle. \end{equation}
Putting \eqref{eq: conv clip prob part 1}, \eqref{eq: conv clip prob part 2}, \eqref{eq: conv clip prob part 3}, \eqref{eq: conv clip prob part 4}  in \eqref{eq: conv clip prob main eq}, we get with probability at least $1-\delta$ that    
\begin{align}\label{overall sum from grad clip high prob}    
    f(\overline{x}_T) - f(x^*) &\leq 2M_2\tau + \left( \frac{\sigma_{q}^{1+\kappa}}{c^\kappa} + \Delta\frac{\sqrt{d}}{\tau}\right)\mathcal{D}_\Psi+\frac12  \frac{R_{0}^{2} }{\nu T} \notag \\
    &+ \frac{\nu}{2} \left[12\sigma_{q}^{1+\kappa} c^{1 - \kappa} \log\left( \frac4\delta\left[ \log\left(\sqrt{T}\right)+2\right]^2\right)\right] \notag \\
    &+ \frac{\nu}{2}\frac{20}{T}c^2\log\left(\frac{12}{\delta}\right) +  \frac{4c\mathcal{D}_\Psi}{T}\log\frac1\delta + \frac{\sqrt{4\sigma_{q}^{1+\kappa} c^{1 - \kappa}}}{\sqrt{T}} \mathcal{D}_\Psi^2 \sqrt{2\log \frac1\delta}. \end{align}
    Next we select optimal parameter in order to minimize upper bound.
    Choosing $c = T^{\frac{1}{(1+\kappa)}}\sigma_{q}$ and putting it in \eqref{overall sum from grad clip high prob}, we get
  \begin{align}\label{after c high prob}
         f(\overline{x}_T) - f(x^*) &\leq 2M_2\tau + \left( \frac{\sigma_{q}}{T^{\frac{\kappa}{1+\kappa}}} + \Delta\frac{\sqrt{d}}{\tau}\right)\mathcal{D}_\Psi +\frac12  \frac{R_0^{2} }{\nu T} \notag \\
    &+ \frac\nu2\left[12\sigma_{q}^{2} T^{\frac{1-\kappa}{(1+\kappa)}} \log\left( \frac4\delta\left[ \log\left(\sqrt{T}\right)+2\right]^2\right)\right] \notag \\ &+\frac\nu2\frac{20\sigma_{q}^{2}}{T^\frac{\kappa-1}{1+\kappa}}\log\left(\frac{12}{\delta}\right) + \frac{4\sigma_{q} \mathcal{D}_\Psi}{T^{\frac{\kappa }{1+\kappa}}}\log\frac1\delta + \frac{2\sigma_{q} }{T^{\frac{\kappa}{1+\kappa}}} \mathcal{D}_\Psi \sqrt{2\log \frac1\delta}.
        \end{align}
    Then we define $\widetilde{\delta}^{-1} = \frac4\delta\left[ \log\left(\sqrt{T}\right)+2\right]^2$, choose  $\nu = \frac{\mathcal{D}_\Psi}{c}$, put it in \eqref{after c high prob} and obtain
    \begin{align}\label{grad clip high prob sum after nu and c}   
    f(\overline{x}_T) - f(x^*) 
    &\leq 2M_2\tau + \left( \frac{\sigma_{q}}{T^{\frac{\kappa}{1+\kappa}}} + \Delta\frac{\sqrt{d}}{\tau}\right)\mathcal{D}_\Psi + \frac {\mathcal{D}_\Psi\sigma_{q}}{2T^{\frac{\kappa }{1+\kappa}}} \left[1 + 12\log\frac{1}{\widetilde{\delta}} + 20\log\frac4\delta\right] \notag \\
      &+\frac{4\sigma_{q} \mathcal{D}_\Psi}{T^{\frac{\kappa }{1+\kappa}}}\log\frac1\delta + \frac{2\sigma_{q} }{T^{\frac{\kappa}{1+\kappa}}} \mathcal{D}_\Psi \sqrt{2\log \frac1\delta}.\end{align}
   Simplifying \eqref{grad clip high prob sum after nu and c},  we get
\begin{align*}
    f(\overline{x}_T) - f(x^*) &\leq 2M_2\tau  + \Delta\frac{\sqrt{d}}{\tau} \mathcal{D}_\Psi \notag \\
    &+ \frac {\mathcal{D}_\Psi\sigma_{q}}{2T^{\frac{\kappa }{1+\kappa}}} \left[3 + 8\log\frac1\delta + 12\log\frac{1}{\widetilde{\delta}} + 20\log\frac4\delta\ + 4\sqrt{2\log\frac1\delta}\right].
\end{align*}
Finally, we get explicit bound of $\sigma_q$ with Lemma \ref{lem: Jensen for norm}
$$\sigma_q \leq 2\left(\frac{\sqrt{d}}{2^{1/4}} a_{q}M_2\right) + 2\left(\frac{da_{q}\Delta}{\tau}\right),$$
And set optimal $\tau$
$$\tau =  \sqrt{\frac{\sqrt{d}\Delta\mathcal{D}_{\Psi} + 2\beta \mathcal{D}_{\Psi}da_{q}\Delta T^{-\frac{\kappa}{1+\kappa}}}{2M_2}}.$$

\end{proof}

\section{Sketch of Proof of \texttt{ZO-Restarts} Convergence} \label{sec: restarts proofs}
\begin{proof}[Proof of  Theorems \ref{restart Conv robust}, \ref{restart Conv clip}]
In this proof $\widetilde{O}(\cdot)$ denotes $\log d$ factor.

\textbf{Step $1$: \texttt{ZO-RSMD} in Expectation.}

Now $x_0$ in Algorithm \ref{alg:Robust} can be chosen in stochastic way.

Similarly to proof of Theorem \ref{Robust Conv} but with $\nu = \frac{\EE\left[D_{\Psi_{p}}(x^*,x_0)\right]^{\frac{1}{1+\kappa}}}{\sigma_{q} }T^{-\frac{1}{1+\kappa}} $ and bound $R_0 \leq \mathcal{D}_\Psi$ one can get from \eqref{smd before nu}
\begin{equation}\label{212}
    \EE[f(\overline{x}_T)] - f(x^*) \leq 2M_2\tau + \frac{\sqrt{d}\Delta }{\tau}\mathcal{D}_\Psi + 2\EE\left[D_{\Psi_{p}}(x^*,x_0)\right]^{\frac{\kappa}{1+\kappa}}\sigma_{q} T^{-\frac{\kappa}{1+\kappa}}.\end{equation}
Under obligatory condition $\Delta \leq  \frac{\sigma_q^2 \EE\left[D_{\Psi_{p}}(x^*,x_0)\right]^{\frac{\kappa}{1+\kappa}}}{M_2\sqrt{d} T^{\frac{2\kappa}{1+\kappa}}}$ picking $\tau = \frac{\sigma_q\EE\left[D_{\Psi_{p}}(x^*,x_0)\right]^{\frac{\kappa}{1+\kappa}}}{M_2 T^{\frac{\kappa}{1+\kappa}}}$, we  obtain from \eqref{212} estimate

\begin{equation}\label{780}
    \EE[f(\overline{x}_T)] - f(x^*) \leq   \left(2 + 1 + 2\right) \frac{\sigma_q \EE\left[D_{\Psi_{p}}(x^*,x_0)\right]^{\frac{\kappa}{1+\kappa}}}{T^{\frac{\kappa}{1+\kappa}}}.
\end{equation}
In $\sigma_{q}$ $\tau$-depending term  has $T^{\frac{-2\kappa}{1+\kappa}}$ decreasing rate, so we neglect it. Next, let's use fact that $D_{\Psi_{p}}(x^*,x_0) = \widetilde{O}(\|x_0 -x^*\|_{p}^{\frac{1+\kappa}{\kappa}})$ from  \cite{gasnikov2018universal}(Remark 3) and denote $R_k = \EE\left[\|\overline{x}_k - x^*\|^\frac{1+\kappa}{\kappa}_{p}\right]^{\frac{\kappa}{1 + \kappa}}$ .

Under $r$-growth Assumption \ref{r-growth} we bound $\EE[f(\overline{x}_T)] - f(x^*)$ from both sides
\begin{equation}\label{two terms sum from restart} \frac{\mu_r
}{2}\EE\left[\|\overline{x}_T - x^*\|_{p}^r\right] \leq\EE[f(\overline{x}_T)] - f(x^*) \leq \widetilde{O}\left(R_0 \frac{\sigma_q}{T^{\frac{\kappa}{1+\kappa}}}\right).
\end{equation}
Due to Jensen's inequality which we can apply since $r \geq \frac{1+\kappa}{\kappa}$ we rewrite \eqref{two terms sum from restart} in order to obtain $R_1$ in it as 
\begin{equation}\label{after all choses}
    \frac{\mu_r
}{2}\EE\left[\|\overline{x}_T - x^*\|^\frac{1+\kappa}{\kappa}_{p}\right]^{r/\frac{1+\kappa}{\kappa}} \leq \frac{\mu_r
}{2}\EE\left[\|\overline{x}_T - x^*\|_{p}^r\right] \leq \widetilde{O}\left(R_0 \frac{\sigma_q}{T^{\frac{\kappa}{1+\kappa}}}\right).\end{equation}
Let's find out after how many iterations $R_0$ value halves
\begin{equation}\label{r halves}
    \frac{\mu_r}{2}R_1^{r} \leq \widetilde{O}\left(R_0 \frac{\sigma_q}{T^{\frac{\kappa}{1+\kappa}}}\right) \leq \frac{\mu_r}{2}\left(\frac{R_0}{2}\right)^{r}. \end{equation}
From right inequality of \eqref{r halves} we obtain number of iterations for one stage 
$$T_1 \geq \widetilde{O}\left(\left(\frac{2^{(1+r)}\sigma_{q}}{\mu_r}\right)^{\frac{1+\kappa}{\kappa}} \frac{1}{R_0^{\frac{(r-1)(1+\kappa)}{\kappa}}}\right).$$
For convenience we define $A \eqdef \frac{2^{(1+r)}\sigma_{q}}{\mu_r} $.

After $T_1$ iterations we restart algorithm with starting point $x_0 = \overline{x}_{T_1}$ and $R_k = R_{k-1}/2= R_0/2^k$. 

After $N$ restarts total number of iterations $T$ will be
\begin{eqnarray}
    \label{Tk}
    T = \sum \limits_{k=1}^{N} T_k &=& \widetilde{O}\left(\frac{A^{\frac{1+\kappa}{\kappa}}}{R_0^{\frac{(r-1)(1+\kappa)}{\kappa}}} \sum\limits_{k=0}^{N-1} 2^{k\left(\frac{(r-1)(1+\kappa)}{\kappa}\right)} \right) \notag \\
    &=&\widetilde{O}\left(\frac{A^{\frac{(1+\kappa)}{\kappa}}}{R_0^{\frac{(r-1)(1+\kappa)}{\kappa}}}  \left[2^{N\left(\frac{(r-1)(1+\kappa)}{\kappa}\right)} - 1\right]\right).\label{eq:geom_sum_of_T_k}
\end{eqnarray}
On the last stage we can get bound with number of restarts $N$ in it
\begin{eqnarray}\EE[f(x_{ \rm final})] - f(x^*) \leq \varepsilon  &=& \widetilde{O}\left(R_{N-1} \frac{\sigma_q}{T_{N}^{\frac{\kappa}{1+\kappa}}}\right) \notag \\
&\leq & \widetilde{O}\left(\frac{\mu_r}{2}\left(\frac{R_{N-1}}{2}\right)^{r}\right) \leq \widetilde{O}\left(\frac{\mu_r}{2}\frac{R_0^r}{2^{(N-1)r}}\right). \notag
\end{eqnarray}
Consequently, in order to get $\varepsilon$ accuracy  we need $N$ restarts and total number of iterations $T$, where  
\begin{equation}\label{total restarts}    
N =\widetilde{O}\left(\frac{1}{r}\log_2\left(\frac{\mu_rR_0^r}{2\varepsilon}\right)\right), 
\end{equation}

\begin{equation}\label{total iter}
    T = \widetilde{O}\left( \left[ \frac{2^\frac{r^2 + 1}{r}\sigma_q}{\mu_r^{1/r}} \cdot \frac{1}{\varepsilon^{\frac{(r-1)}{r}}}\right]^\frac{1+\kappa}{\kappa} \right), \quad T_k = \widetilde{O}\left( \left[ \frac{\sigma_q 2^{(1+r)}}{\mu_r R_0^{r-1}} 2^{k(r-1)}\right]^\frac{1+\kappa}{\kappa} \right).  \end{equation}

In each restart section we get different bounds for noise absolute value. From  $T_k$ formula from \eqref{Tk} we get bound
\begin{equation}\label{Deltak}
    \Delta_k = \widetilde{O}\left(\frac{\mu_r^2 R_0^{(2r - 1)}}{M_2 \sqrt{d}} \frac{1}{2^{k(2r-1)}}\right) .
\end{equation}
Hence, $\Delta_k$ will be the smallest on the last iteration, when $k = N$, i.e.
$$\Delta_N =  \widetilde{O}\left(\frac{\mu_r^{1/r} }{M_2 \sqrt{d}} \varepsilon^{(2-1/r)}\right).$$
\textbf{Step $2$: \texttt{ZO-Clip-SMD} in Expectation.}

Now $x_0$ in Algorithm \ref{alg:clip} can be chosen in stochastic way.

Similarly to proof of Theorem \ref{Clip Conv}  but with $\nu^* = \EE\left[D_{\Psi_{p}}(x^*,x_0)\right]^\frac{1}{2} \left( \frac{1}{4T\sigma_{q}^{1+\kappa} }\right)^{\frac{1}{1+\kappa}}, c^* = {\EE\left[D_{\Psi_{p}}(x^*,x_0)\right]^\frac{1}{2}}/{\nu^*} $  one can get from \eqref{clip before nu}
\begin{equation}\label{213}
    \EE[f(\overline{x}_T)] - f(x^*) \leq 2M_2\tau + \Delta\frac{\sqrt{d}}{\tau} \mathcal{D}_\Psi+2 \frac{\sigma_{q}\EE\left[D_{\Psi_{p}}(x^*,x_0)\right]^\frac{1}{2}  }{T^\frac{\kappa}{1+\kappa}}.\end{equation}
Under obligatory condition $\Delta \leq \frac{\sigma_q^2 \EE\left[D_{\Psi_{p}}(x^*,x_0)\right]^{\frac{1}{2}}}{M_2\sqrt{d} T^{\frac{2\kappa}{1+\kappa}}}$ picking $\tau = \frac{\sigma_q\EE\left[D_{\Psi_{p}}(x^*,x_0)\right]^{\frac{1}{2}}}{M_2 T^{\frac{\kappa}{1+\kappa}}}$, we  obtain from \eqref{213} estimate

\begin{equation}\notag
    \EE[f(\overline{x}_T)] - f(x^*) \leq \left(2 + 1+ 2\right)\frac{\sigma_q \EE\left[D_{\Psi_{p}}(x^*,x_0)\right]^\frac{1}{2}}{T^{\frac{\kappa}{1+\kappa}}}.
\end{equation}
In $\sigma_{q}$ $\tau$-depending term  has $T^{\frac{-2\kappa}{1+\kappa}}$ decreasing rate, so we neglect it. Next, let's use fact that $D_{\Psi_{p}}(x^*,x_0) = \widetilde{O}(\|x_0 -x^*\|_{p}^2)$ from  \cite{gasnikov2018universal}(Remark 3) and denote $R_k = \EE\left[\|\overline{x}_k - x^*\|^2_{p}\right]^{\frac{1}{2}}$.

Under $r$-growth Assumption \ref{r-growth} we bound  $\EE[f(\overline{x}_T)] - f(x^*)$ from both sides
$$\frac{\mu_r
}{2}\EE\left[\|\overline{x}_T - x^*\|_{q^*}^r\right] \leq\EE[f(\overline{x}_T)] - f(x^*) \leq \widetilde{O}\left( R_0\frac{\sigma_q}{T^{\frac{\kappa}{(1+\kappa)}}}\right).$$
Due to Jensen's inequality which we can apply since $r \geq 2$ we  obtain
$$\frac{\mu_r
}{2}\EE\left[\|\overline{x}_T - x^*\|^2_{q^*}\right]^{r/2} \leq \frac{\mu_r
}{2}\EE\left[\|\overline{x}_T - x^*\|_{q^*}^r\right] \leq \widetilde{O}\left(R_0\frac{\sigma_q}{T^{\frac{\kappa}{(1+\kappa)}}}\right).$$
Next part of the proof is the same from \textbf{Step $1$} starting from \eqref{after all choses}.
Analogically, we get the same $T_2$,$N_2$ and noise bounds from \eqref{total iter}, \eqref{total restarts} and \eqref{Deltak} correspondingly.

\textbf{Step $3$: \texttt{ZO-Clip-SMD} in High Probability.}

Now $x_0$ in Algorithm \ref{alg:clip} can be chosen in stochastic way.

Important moment about convergence in high probability in restart setup is to control final probability. Let number of restarts be $N_3$,  if each restart has probability to be in bounds at least $1 -\delta/N_3$ then final probability to be in bounds will be greater than $1 - \delta$ which is probability of  'all restarts to be in bounds'. Usually $N_3 \sim \log(\frac1\varepsilon)$, thus $$\log\frac{N_3}{1} = \log\log \frac1\varepsilon \ll \log\frac{1}{\delta} \frac{1}{\varepsilon^\frac{1+\kappa}{\kappa}}.$$
It means that we can use $\log\frac1\delta$ instead of $\log\frac{N_3}\delta$.

Similarly to proof of Theorem \ref{Clip Conv prob}  but  $\nu^* = \left[D_{\Psi_{p}}(x^*,x_0)\right]^{1/2} \left( \frac{1}{T\sigma_{q}^{1+\kappa} }\right)^{\frac{1}{1+\kappa}}, c^* = {\EE\left[D_{\Psi_{p}}(x^*,x_0)\right]^\frac{1}{2}}/{\nu^*} $  one can get from \eqref{after c high prob} with probability at least $1-\delta/N_3$
\begin{align*}
   f(\overline{x}_T) - f(x^*) &\leq 2M_2\tau  + \Delta\frac{\sqrt{d}}{\tau} \mathcal{D}_\Psi \\
   &+\frac { \left[D_{\Psi_{p}}(x^*,x_0)\right]^{1/2}\sigma_{q}}{2T^{\frac{\kappa }{1+\kappa}}} \left[3 + 8\log\frac1\delta + 12\log\frac{1}{\widetilde{\delta}} + 20\log\frac4\delta\ + 4\sqrt{2\log\frac1\delta}\right]. 
\end{align*}
Denote $\widetilde{\delta}^{-1} = \frac4\delta\left[ \log\left(\sqrt{T}\right)+2\right]^2, \beta = \left[3 + 8\log\frac1\delta + 12\log\frac{1}{\widetilde{\delta}} + 20\log\frac4\delta\ + 4\sqrt{2\log\frac1\delta}\right]$. 

Under obligatory condition $\Delta \leq \frac{\beta^2 \sigma_q^2 D^{\frac{1}{2}}_{\Psi_{p}}(x^*,x_0)}{M_2\sqrt{d} T^{\frac{2\kappa}{1+\kappa}}}$ picking $\tau = \frac{\beta \sigma_q D^{\frac{1}{2}}_{\Psi_{p}}(x^*,x_0)}{M_2 T^{\frac{\kappa}{1+\kappa}}}$, we  obtain estimate
    $$f(\overline{x}_T) - f(x^*) \leq\left(2 + 1 + 1\right)\frac{\sigma_q \beta\left[D_{\Psi_{p}}(x^*,x_0)\right]^\frac{1}{2}}{T^{\frac{\kappa}{1+\kappa}}}.$$
In $\sigma_{q}$ $\tau$-depending term  has $T^{\frac{-2\kappa}{1+\kappa}}$ decreasing rate, so we neglect it. Next, let's use fact that $D_{\Psi_{p}}(x^*,x_0) = \widetilde{O}(\|x_0 -x^*\|_{p}^2)$ from  \cite{gasnikov2018universal}(Remark 3) and denote $R_k = \|\overline{x}_k - x^*\|_{p}$.

Under $r$-growth Assumption \ref{r-growth} we get
$$\frac{\mu_r}{2}\|\overline{x}_T - x^*\|_{p}^r \leq f(\overline{x}_T) - f(x^*) \leq \widetilde{O}\left( R_0\frac{\sigma_q \beta}{T^{\frac{\kappa}{(1+\kappa)}}}\right).$$

For $r>1$ next part of the proof is the same from \textbf{Step $1$} starting from \eqref{after all choses} with 
$$A \eqdef \frac{2^{(1+r)}\beta \sigma_{q}}{\mu_r}. $$
Analogically, we get $T_3$,$N_3$ and noise bounds from \eqref{total iter}, \eqref{total restarts} and \eqref{Deltak} correspondingly.
\begin{equation}\notag
N =\widetilde{O}\left(\frac{1}{r}\log_2\left(\frac{\mu_rR_0^r}{2\varepsilon}\right)\right), 
\end{equation}

\begin{equation}\label{total iter step 3}
    T = \widetilde{O}\left( \left[ \frac{2^\frac{r^2 + 1}{r}\sigma_q \beta}{\mu_r^{1/r}}  \frac{1}{\varepsilon^{\frac{(r-1)}{r}}}\right]^\frac{1+\kappa}{\kappa} \right), \quad T_k = \widetilde{O}\left( \left[ \frac{\sigma_q \beta 2^{(1+r)}}{\mu_r R_0^{r-1}} 2^{k(r-1)}\right]^\frac{1+\kappa}{\kappa} \right).  \end{equation}
In each restart section we get different bounds for noise absolute value. From  $T_k$ formula from \eqref{total iter step 3}
\begin{equation}\notag
    \Delta_k = \widetilde{O}\left(\frac{\mu_r^2 R_0^{(2r - 1)}}{M_2 \sqrt{d}} \frac{1}{2^{k(2r-1)}}\right) .
\end{equation}
Hence, $\Delta_k$ will be the smallest on the last iteration, when $k = N$, i.e. 
$$\Delta_N =  \widetilde{O}\left(\frac{\mu_r^{1/r} }{M_2 \sqrt{d}} \varepsilon^{(2-1/r)}\right).$$

In case of $r = 1$ number of iterations $T_k$ at step $k$ from \eqref{total iter step 3} changes as
$$T_k = A^{\frac{1+\kappa}{\kappa}},$$
and we do not need to apply the formula \eqref{eq:geom_sum_of_T_k} for the sum of the geometric progression. Thus, total number of iterations $T$ after $N = \widetilde{O}\left(\log_2\left(\frac{\mu_rR_0}{2\varepsilon}\right)\right)$ restarts equals
$$T = \sum \limits_{k=1}^N T_k = N A^{\frac{1+\kappa}{\kappa}} = \widetilde{O}\left(\left[\frac{\beta \sigma_{q}}{ \mu_r}\right]^{\frac{1+\kappa}{\kappa}} \log_2\left(\frac{\mu_rR_0}{2\varepsilon}\right)\right).$$

\end{proof}
\end{document}